\documentclass[11pt,reqno]{amsart}
\usepackage[T1]{fontenc}
\usepackage{lmodern}
\usepackage[letterpaper,margin=1in]{geometry}
\usepackage{mathtools,amssymb,mathrsfs}
\usepackage{tikz-cd}
\usepackage{enumitem,booktabs,array,microtype,needspace}
\microtypesetup{expansion=false}
\usepackage[hidelinks]{hyperref}
\hypersetup{
  pdftitle={Stability of representation cohomology sheaves on moduli spaces of bundles},
  pdfauthor={Jianping Wang and Xueqing Wen}
}
\newtheorem{theorem}{Theorem}[section]
\newtheorem{proposition}[theorem]{Proposition}
\newtheorem{lemma}[theorem]{Lemma}

\theoremstyle{definition}
\newtheorem{definition}[theorem]{Definition}
\newtheorem{question}[theorem]{Question}
\theoremstyle{remark}
\newtheorem{remark}[theorem]{Remark}
\numberwithin{equation}{section}
\mathtoolsset{showonlyrefs}
\newcommand{\C}{\mathbb C}
\newcommand{\Pp}{\mathbb P}
\newcommand{\OO}{\mathcal O}
\newcommand{\EE}{\mathcal E}

\newcommand{\FF}{\mathcal F}
\newcommand{\MM}{\mathfrak M}

\newcommand{\gfr}{\mathfrak g}
\newcommand{\hfr}{\mathfrak h}

\newcommand{\rk}{\operatorname{rk}}
\newcommand{\im}{\operatorname{im}}
\newcommand{\Ann}{\operatorname{Ann}}
\newcommand{\ad}{\operatorname{ad}}
\newcommand{\Ad}{\operatorname{Ad}}
\newcommand{\End}{\operatorname{End}}

\newcommand{\Sym}{\operatorname{Sym}}
\newcommand{\SL}{\operatorname{SL}}
\newcommand{\GL}{\operatorname{GL}}
\newcommand{\PGL}{\operatorname{PGL}}

\newcommand{\Bun}{\operatorname{Bun}}
\newcommand{\Pic}{\operatorname{Pic}}
\newcommand{\Cl}{\operatorname{Cl}}
\newcommand{\SU}{\mathcal{SU}}
\newcommand{\Lie}{\operatorname{Lie}}
\newcommand{\tr}{\operatorname{tr}}
\newcommand{\Res}{\operatorname{Res}}
\newcommand{\codim}{\operatorname{codim}}

\newcommand{\eps}{\varepsilon}
\newcommand{\etale}{\'{e}tale}

\setlist[enumerate,1]{label=\textup{(\roman*)},leftmargin=2.1em,itemsep=4pt,topsep=4pt}
\allowdisplaybreaks[1]
\title[Stability of representation cohomology sheaves]{Stability of representation cohomology sheaves on moduli spaces of bundles}

\author{Jianping Wang}
\address{Beijing Institute of Mathematical Sciences and Applications, No. 544, Hefangkou Village, Huaibei Town, Huairou District, Beijing, 101408, China}
\email{jianpw@bimsa.cn}

\author{Xueqing Wen}
\address{Chongqing University of Technology, No. 69, Hongguang Avenue, Banan District, Chongqing, 400054, China.}
\email{wenxq@cqut.edu.cn}

\date{}

\begin{document}
\begin{abstract}
We prove that the tangent sheaf of $\SU_C(n,L)$ is slope stable for every
smooth complex projective curve $C$ of genus $g\geq2$, every $n\geq2$,
and every determinant $L$. For $g\geq3$, we prove a more general stability
theorem on moduli spaces of principal bundles for connected reductive
groups, using cohomology sheaves associated with representations in
place of tangent sheaves. The adjoint representation gives tangent
sheaf stability for the case of almost simple groups. The moduli spaces may be
singular, and stability holds with respect to every ample Cartier
polarization.
\end{abstract}
\maketitle

\section{Introduction}\label{sec:intro}

The slope stability of the tangent bundle is an important question in the geometry of Fano varieties. For a smooth Fano variety, $K$-polystability is equivalent to the existence of a K\"ahler--Einstein metric \cite{Ber16,CDS15III,Tia15}. Such a metric induces a Hermitian--Einstein metric on the tangent bundle and hence slope polystability with respect to the anticanonical polarization \cite{Kob82,Lub83,Don87,UY86}.

For a smooth complex projective curve $C$, the moduli spaces of semistable vector bundles with fixed determinant, denoted by $\SU_C(n,L)$, form a distinguished class of normal Fano varieties of Picard number one. They are smooth when $n$ and $\deg L$ are coprime \cite{Ram73,DN89}. The K\"ahler--Einstein problem for these spaces has a long history. Nadel's example, a smooth intersection of two quadrics in $\Pp^5$ which admits a K\"ahler--Einstein metric \cite{Nad90}, is isomorphic to $\SU_C(2,L)$ for some curve $C$ of genus $2$, described by Newstead \cite{New68}. More recently, Keller proved that $\SU_C(2,L)$ is $K$-stable for a general curve $C$ with $g\geq 2$ and $L$ of odd degree \cite[Theorem~1.1]{Kel24}. For a recent explicit formulation of the $K$-stability problem for $\SU_C(n,L)$, see Xu--Zhuang \cite[Problem~9(1)]{XZ26}.

We study the slope stability of the tangent sheaf of $\SU_C(n,L)$.
For rank two and odd degree, Hwang proved stability by studying Hecke
curves and their tangent directions \cite{Hwa00}. His approach relates the problem to varieties of minimal rational tangents. Further results on Hecke curves and on minimal rational curves were obtained by Hwang--Ramanan \cite{HR04} and Sun
\cite{Sun05}. Stability in higher rank, for coprime rank and degree and genus at least three, was also stated in an
unpublished preprint of Iyer \cite{Iye14}. Our proof is independent of that preprint and establishes tangent sheaf stability for every
$n\geq2$, $g\geq2$, and for every determinant $L$.

Without the coprimality assumption, these moduli spaces may be singular. For a klt $\mathbb Q$-Fano variety $X$, $K$-semistability implies
slope semistability of $T_X$ with respect to $-K_X$, while the
existence of a K\"ahler--Einstein metric implies slope polystability
\cite{Li21,DGP24}.
These results also motivate the study of tangent sheaves on possibly singular moduli spaces such as moduli spaces of principal bundles on a curve.

More generally, we consider moduli spaces of principal bundles for a
connected reductive group $G$, with or without a fixed induced
$G/G_{\mathrm{der}}$-bundle. We study the cohomology sheaves associated
with representations of $G$. The adjoint representations specified
below give the tangent sheaves.

Let $C$ be a smooth connected complex projective curve of genus
$g\geq2$. For a connected reductive complex algebraic group $G$, we
use the notation
\[
 G_{\mathrm{der}}=[G,G],\qquad
 G_{\mathrm{ad}}=G/Z(G),\qquad
 G^{\mathrm{ab}}=G/G_{\mathrm{der}}.
\]
For $\tau\in\pi_1(G)$, let $\MM_G^\tau=\Bun_G^\tau(C)$ be the stack
of principal $G$-bundles of topological type $\tau$, and let
$M_G^\tau(C)$ be the moduli space of S-equivalence classes of
semistable objects, as constructed by Ramanathan
\cite{Ram96I,Ram96II}. Write $\tau^{\mathrm{ab}}$ for the image of
$\tau$ in $\pi_1(G^{\mathrm{ab}})$. Extension of structure group
induces
\[
 \operatorname{ab}:M_G^\tau(C)\longrightarrow
 M_{G^{\mathrm{ab}}}^{\tau^{\mathrm{ab}}}(C),\qquad
 [P]\longmapsto[P\times^G G^{\mathrm{ab}}].
\]
For a principal $G^{\mathrm{ab}}$-bundle $Q$ of type
$\tau^{\mathrm{ab}}$, let $\MM_{G,Q}^{\tau}$ be the stack of pairs
$(P,\alpha)$, where
$\alpha:P\times^G G^{\mathrm{ab}}\xrightarrow{\sim}Q$, and put
\[
 M_{G,Q}^{\tau}(C)=\operatorname{ab}^{-1}([Q]).
\]
This is the moduli space of S-equivalence classes of semistable
objects whose induced $G^{\mathrm{ab}}$-bundle is isomorphic to $Q$. When $G$ is semisimple,
$G^{\mathrm{ab}}=1$ and the two versions coincide. For $G=\GL_n$, $d\in\pi_1(G)$ and
$L\in M_{\C^*}^d(C)=\Pic^d(C)$, we recover $\SU_C(n,L)$:
\[
 \operatorname{ab}=\det:M_{\GL_n}^{d}(C)\longrightarrow\Pic^{d}(C),\qquad
 M_{\GL_n,L}^{d}(C)=\SU_C(n,L),\quad \deg L=d.
\]
Unless a version is specified, $M$ denotes either
$M_G^\tau(C)$ or $M_{G,Q}^{\tau}(C)$, and is assumed to be nonempty.

Let $(\rho,V)$ be a nonzero finite-dimensional algebraic
representation satisfying
\[
\rho:G\longrightarrow\GL(V),\qquad r=\dim V,\qquad
 V^G=0,\qquad \rho|_{Z(G)}=1.
\]
For a principal bundle $P$, write $E_\rho(P)=P\times^G V$.
Write $G_{\mathrm{ad}}=\prod_{i=1}^sG_i$, with $G_i$ adjoint almost
simple, and fix a maximal torus and a Borel subgroup in each factor.
For each $i$, let $m_i$ be the largest highest weight occurring
in the restriction of $V$ to the $\SL_2$ associated with the
highest long root of $G_i$. Put
\begin{equation}\label{eq:m}
 m_\rho=\max_i \{m_i\}.
\end{equation}

For $g\geq3$, let $U\subset M$ be the inverse image of the
regularly stable loci in the adjoint factors, as in
Section~\ref{sec:moduli}. By
Proposition~\ref{prop:associated-descent}, the universal associated
bundle descends to a vector bundle $\EE_\rho$ on $U\times C$.
Let $p:U\times C\to U$ be the projection. Since $M$ is normal
and $\codim_M(M\setminus U)\geq2$, define $\mathcal V_\rho$
to be the unique coherent reflexive sheaf on $M$ extending
$(R^1p_*\EE_\rho)^{**}$ from $U$, so that
\begin{equation}\label{eq:Vr-intro}
 \mathcal V_\rho|_U\simeq(R^1p_*\EE_\rho)^{**}.
\end{equation}
By Theorem~\ref{thm:interp} below, $H^0(C,E_\rho(P))=0$ for general
$[P]\in U$. Cohomology and base change therefore identifies
the fiber of $\mathcal V_\rho$ at such a point with
$H^1(C,E_\rho(P))$, and Riemann--Roch gives
$\rk\mathcal V_\rho=(g-1)r$.

These sheaves differ from the Poincar\'e bundles and their
associated bundles studied in \cite{BBN97}: they are constructed
from first higher direct images along $C$, rather than being
universal bundles on the product with $C$ or their restrictions
to a fixed point of $C$.

\begin{theorem}\label{thm:main}
Let $g\geq3$, and let $(\rho,V)$ be a representation of
$G$ such that $V^G=0$, $\rho|_{Z(G)}=1$, and $m_\rho\leq g-1$.
For either choice of $M$, if $(\rho,V)$ is irreducible, then the sheaf $\mathcal V_\rho$ is slope stable
with respect to every ample Cartier polarization on $M$.
\end{theorem}

Let $\Ad:G\to\GL(\Lie(G))$ denote the adjoint representation.
On $M=M_{G,Q}^{\tau}(C)$, take $V=\Lie(G_{\mathrm{der}})$ with
$\rho(h)=\Ad(h)|_{\Lie(G_{\mathrm{der}})}$ for $h\in G$.
On $M=M_G^\tau(C)$ with $G$ semisimple, take
$(\rho,V)=(\Ad,\Lie(G))$.
In either case, for $g\geq3$, deformation theory identifies
$R^1p_*\EE_\rho$ with $T_U$, and its reflexive extension is
\[
 \mathcal V_\rho\simeq T_M
 =\mathcal Hom_{\OO_M}(\Omega_M^1,\OO_M),
\]
by Proposition~\ref{prop:cohomology-sheaf}.

For the conjugation action of $G$ on $\Lie(G_{\mathrm{der}})$,
the restriction to each highest-root $\SL_2$ has largest highest
weight $2$. So $m_\rho=2$ in this case. When $G_{\mathrm{der}}$ is almost simple, this representation is
irreducible, so Theorem~\ref{thm:main} gives the following result
for $g\geq3$. The vector bundle case in genus $2$ is proved separately
in Section~\ref{sec:examples}.

\begin{theorem}\label{cor:tangent}\label{thm:tangent}
The tangent sheaf $T_M$ is slope stable with respect to every ample Cartier polarization in each of the following cases:
\begin{itemize}
\item[\textup{(i)}]
$g\geq3$ and $M=M_{G,Q}^{\tau}(C)$, where $G$ is connected reductive and $G_{\mathrm{der}}$ is almost simple. Recall that if $G$ is almost simple, then $M_{G,Q}^{\tau}(C)=M_{G}^{\tau}(C)$.
\item[\textup{(ii)}]
$g\geq2$ and $M=\SU_C(n,L)$, where $n\geq2$ and $L$ is any line bundle
on $C$.
\end{itemize}
\end{theorem}

Our proof uses the geometry of Hecke curves, motivated by the works of
Hwang, Hwang--Ramanan, and Sun \cite{Hwa00,Hwa02,HR04,Sun05}.
The Hecke modifications used to construct these curves were
studied by Narasimhan--Ramanan \cite{NR78}; the
use of rational curves in tangent bundle stability is also central
to Hwang's work \cite{Hwa98}.
We explain the idea for $M=M_{G,Q}^{\tau}(C)$. Let
$0\ne\mathcal F\subsetneq\mathcal V_\rho$ be a saturated subsheaf
of smaller rank. Choose a general $[P]\in U$ and use Serre
duality to set
\[
 \begin{gathered}
 Z=\Ann(\mathcal F|_{[P]})
   \subset H^0(C,E_\rho(P)^\vee\otimes K_C), \qquad
 Z^\vee=(\mathcal V_\rho/\mathcal F)|_{[P]},\qquad
 \dim Z=\rk(\mathcal V_\rho/\mathcal F).
 \end{gathered}
\]
After fixing $P$ and $Z$, choose a general $x\in C$ and a complete
highest-root Hecke curve $f:\Pp^1\to U$, with modification point
$x$ and $f(0)=[P]$, such that $H^0(C,E_\rho(P'))=0$ for every
$[P']\in f(\Pp^1)$ and $\mathcal F$, $\mathcal V_\rho$, and
$\mathcal V_\rho/\mathcal F$ are locally free near its image;
see Proposition~\ref{prop:complete-hecke} and
Lemma~\ref{lem:common-containing-bundle}.
We study $c_1(\mathcal V_\rho)$ and
$c_1(\mathcal V_\rho/\mathcal F)$ by pulling back along $f$.
The argument proceeds as follows.

\begin{itemize}[leftmargin=4.5em]
\item[Step 1.]
The splitting type of $f^*\mathcal V_\rho$ can be computed from the
representation $V$. This gives $\deg f^*\mathcal V_\rho$ and shows
that $f^*\mathcal V_\rho$ is globally generated; see
Proposition~\ref{prop:splitting}.

\item[Step 2.]
Thus $f^*(\mathcal V_\rho/\mathcal F)$ is also globally generated.
The positive degree subbundles of $f^*\mathcal V_\rho$ induce on this
quotient the filtration given by
\[
 \im\!\left(
 H^0\bigl(\Pp^1,(f^*\mathcal V_\rho)(-j)\bigr)
 \otimes\OO_{\Pp^1}(j)
 \longrightarrow f^*(\mathcal V_\rho/\mathcal F)
 \right),\qquad j\geq1.
\]

\item[Step 3.]
The images of the corresponding maps on the fiber at $0\in\Pp^1$
give a filtration of $Z^\vee$. The sum of their dimensions gives
the lower bound for $\deg f^*(\mathcal V_\rho/\mathcal F)$ in
\eqref{eq:factor-degree-bound}.

\item[Step 4.]
Finally, we regard $Z$ as a space of global sections on $C$ and
consider its filtration $Z(-px)$, $p\geq0$, by vanishing at $x$.
We relate it to the filtration in the preceding step and apply the
two representation-theoretic Lemmas~\ref{lem:one-factor-strict}
and~\ref{lem:simultaneous-strictness} to obtain the further lower
bound in \eqref{eq:factor-quotient-estimate}, which proves stability.
\end{itemize}

For $g=2$ and $M=\SU_C(n,L)$, we prove tangent-sheaf stability
in Section~\ref{sec:examples}, using elementary modifications
and the action of $\Pic^0(C)[n]$ by tensor product to rule out a
destabilizing subsheaf.

The Hecke calculation requires cohomological vanishing at the
modification point. We record the following theorem, which also has an interpretation in terms of theta
divisors. For a degree zero associated bundle, consider
\[
 \Theta_{E_\rho(P)}=
 \{L'\in\Pic^{g-1}(C):H^0(C,E_\rho(P)\otimes L')\ne0\}.
\]
When this locus is not the entire $\Pic^{g-1}(C)$, it is the theta divisor of $E_\rho(P)$; see \cite[\S1.3 and \S2.1]{Bea06}. The theorem below shows that, for general $P$, the image of
\[
 C\longrightarrow\Pic^{g-1}(C),\qquad  x\longmapsto\OO_C((g-1)x)
\]
meets $\Theta_{E_\rho(P)}$ in only finitely many points.

\begin{theorem}[Theorem \ref{thm:vanishing}]\label{thm:interp}
Let $g\geq2$, let $G$ be connected reductive, and let $(\rho,V)$ be a finite-dimensional algebraic representation satisfying $V^G=0$ and $ \rho|_{Z(G)}=1$. For every $\tau$ and $Q$ as above, a general pair $(P,x)$ in either $\MM_G^\tau\times C$ or $\MM_{G,Q}^{\tau}\times C$ satisfies
\begin{equation}\label{eq:interp}
 H^0\bigl(C,E_\rho(P)((g-1)x)\bigr)  =H^1\bigl(C,E_\rho(P)((g-1)x)\bigr)=0.
\end{equation}
\end{theorem}

The genus $2$ result for $\SU_C(n,L)$ shows that
$m_\rho\leq g-1$ is not necessary in general: the adjoint
representation has $m_\rho=2>g-1$ when $g=2$. This motivates the
following question, where $\mathcal V_\rho$ is defined by
\eqref{eq:Vr-intro} without a bound on $m_\rho$. We retain $g\geq3$ to ensure that
$\codim_M(M\setminus U)\geq2$.

\begin{question}\label{question:all-representations}
Let $g\geq3$, and let $M$ be either of the two moduli spaces above.
For a nontrivial irreducible representation $(\rho,V)$ of $G$
trivial on $Z(G)$, is $\mathcal V_\rho$ slope stable with respect to
every ample Cartier polarization, without a bound on $m_\rho$?
\end{question}

\noindent\textbf{Organization of the paper.}
Section~\ref{sec:moduli} describes the moduli spaces, constructs
the universal associated bundle on $U\times C$, and defines
$\mathcal V_\rho$ on both moduli spaces.
Section~\ref{sec:interpolation} proves Theorem~\ref{thm:interp} and
the monotonicity of vanishing increments.
Section~\ref{sec:hecke} constructs complete Hecke curves on
$M_{G,Q}^{\tau}(C)$, computes the restriction of
$\mathcal V_\rho$ to these curves, and determines its first
Chern class on this moduli space.
Section~\ref{sec:slope} proves Theorem~\ref{thm:main}, first for
$M=M_{G,Q}^{\tau}(C)$ and then for $M=M_G^\tau(C)$; its final
subsection contains the two representation-theoretic lemmas
used in the proof.
Section~\ref{sec:examples} proves tangent-sheaf stability in
genus two for fixed-determinant vector bundles.

\subsection*{Acknowledgements}

The second named author is grateful to Xucheng Zhang for carefully reading an earlier version of this paper and for his valuable comments.

\smallskip\noindent
\textbf{AI disclosure.}
AI tools were used in the course of this research to explore proof strategies, assist in the development and refinement of mathematical arguments, identify relevant references, and improve the exposition. All AI-assisted material included in the paper was checked by the authors, who take full responsibility for its contents.

\section{Moduli spaces and cohomology sheaves}\label{sec:moduli}

\subsection{The moduli space \texorpdfstring{$M_{G,Q}^{\tau}(C)$}{M(G,Q,tau)(C)} and the adjoint quotient}

Keep the notation of the introduction, and assume $g\geq3$ unless
otherwise stated. The stack $\MM_{G,Q}^{\tau}$ parametrizes pairs
$(P,\alpha)$, where $P$ has topological type $\tau$ and
\[
 \alpha:P\times^G G^{\mathrm{ab}}\xrightarrow{\sim}Q.
\]
Its semistable open substack has coarse
moduli space $M_{G,Q}^{\tau}(C)$. We write
\[
 \MM=\MM_{G,Q}^{\tau},\qquad M=M_{G,Q}^{\tau}(C).
\]
For semisimple $G$, we omit $Q$ and use
$\MM_G^\tau=\Bun_G^\tau(C)$.

Recall that $G_{\mathrm{ad}}=\prod_{i=1}^sG_i$, with each $G_i$
adjoint and almost simple. Since $\rho|_{Z(G)}=1$, the representation
factors through $G_{\mathrm{ad}}$; we use the same notation $\rho$
for this factorization. If $V$ is irreducible, choose irreducible
representations $\rho_i:G_i\to\GL(V_i)$ and a
$G_{\mathrm{ad}}$-equivariant identification
\begin{equation}\label{eq:tensor-decomp-intro}
 V\simeq\bigotimes_{i=1}^sV_i,\qquad
 \rho(g_1,\ldots,g_s)=\bigotimes_{i=1}^s\rho_i(g_i).
\end{equation}
A factor $V_i$ is allowed to be the one-dimensional trivial
representation. Irreducibility of $V$ is not required for the other
constructions in this section.

For $(P,\alpha)\in\MM_{G,Q}^{\tau}$, the adjoint action of $G$
preserves $\Lie(G_{\mathrm{der}})$. We write
\[
 E_{\Ad}(P):=P\times^G\Lie(G_{\mathrm{der}})
\]
for the bundle associated with this action.

Let $\tau_i\in\pi_1(G_i)$ be the type induced by $\tau$. Extension of
structure group gives morphisms
\begin{equation}\label{eq:adjoint-map}
 \Phi:\MM\longrightarrow\prod_{i=1}^s\MM_{G_i}^{\tau_i},
 \qquad (P,\alpha)\longmapsto(P\times^G G_i)_{i=1}^s,
\end{equation}
and, on semistable coarse moduli spaces,
\begin{equation}\label{eq:coarse-adjoint-map}
 \phi:M\longrightarrow\prod_{i=1}^sM_{G_i}^{\tau_i}(C).
\end{equation}
Let $U\subset M$ be the inverse image of the product of the regularly
stable loci. Since the groups $G_i$ are adjoint, each corresponding
$G_i$-bundle has trivial automorphism group. Write
$\mathfrak U\subset\MM$ for the corresponding open substack.

\begin{proposition}\label{prop:inputs}
The stack $\MM$ is smooth and irreducible, and $M$ is irreducible,
normal and projective, with
\[
 \dim M=(g-1)\dim G_{\mathrm{der}}.
\]
The morphism $\Phi$ is surjective on the specified components, and
$\phi$ is finite and surjective. The open set $U$ is smooth,
$\mathfrak U\to U$ is a gerbe with band $Z(G_{\mathrm{der}})$, and
\[
 \codim_M(M\setminus U)\geq2,
 \qquad \codim_{\MM}(\MM\setminus\mathfrak U)\geq2.
\]
There are nef rational Cartier classes
$\Theta_1,\ldots,\Theta_s$ forming a basis of $\Cl(M)_{\mathbb Q}$.
For every ample Cartier divisor $H$ on $M$,
\begin{equation}\label{eq:theta-positive}
 \Theta_i\cdot H^{\dim M-1}>0.
\end{equation}
\end{proposition}

\begin{proof}
Put
\[
 T:=G^{\mathrm{ab}},\qquad
 \mu:=Z(G_{\mathrm{der}}),\qquad
 \mathcal N:=\prod_{i=1}^s\MM_{G_i}^{\tau_i},
 \qquad
 N:=\prod_{i=1}^sM_{G_i}^{\tau_i}(C).
\]
Consider the central isogeny
\begin{equation}\label{eq:central-isogeny}
 1\longrightarrow\mu\longrightarrow G
 \longrightarrow T\times G_{\mathrm{ad}}
 \longrightarrow1.
\end{equation}
By \cite[Example~5.1.4]{BH10}, the induced morphism
\[
 \MM_G^\tau
 \longrightarrow
 \MM_T^{\tau^{\mathrm{ab}}}\times\mathcal N
\]
is a torsor under the group stack $\Bun_\mu(C)$.
Its base change along
\[
 \mathcal N\longrightarrow
 \MM_T^{\tau^{\mathrm{ab}}}\times\mathcal N,
 \qquad
 P_{\mathrm{ad}}\longmapsto(Q,P_{\mathrm{ad}})
\]
is precisely $\Phi:\MM\to\mathcal N$.
Consequently, $\Phi$ is also a $\Bun_\mu(C)$-torsor.
Since $\mu$ is a finite abelian group over $\C$, the stack
$\Bun_\mu(C)$ is smooth of dimension zero.
Thus $\Phi$ is surjective and smooth of relative dimension zero.
The stack $\mathcal N$ is smooth, so $\MM$ is smooth as well.

 To apply the twisted uniformization in
\cite[\S2.4]{BLS98}, use the presentation
\[
 G\simeq
 \bigl(G_{\mathrm{der}}\times Z(G)^0\bigr)
 \big/
 \{(z,z^{-1}):z\in G_{\mathrm{der}}\cap Z(G)^0\}.
\]
The isogeny $Z(G)^0\to G^{\mathrm{ab}}$ induces a surjection on
isomorphism classes of degree-zero bundles on $C$, by divisibility
of $\Pic^0(C)$. Tensoring by a fixed degree-zero $Z(G)^0$-bundle
therefore reduces an arbitrary $Q$ to the torus bundles of the same
type used in \cite[\S2.1]{BLS98}. The uniformization in
\cite[\S2.4]{BLS98}, applied to $G_{\mathrm{der}}$, proves that
$\MM$ is connected, hence irreducible. The construction of the
coarse moduli space gives normality and projectivity of $M$;
see \cite[\S7.1 and Lemma~7.3]{BLS98}. Its irreducibility follows
from that of the semistable open substack. Semistability is
preserved and reflected by \eqref{eq:central-isogeny}. Thus
$\phi$ is surjective and has finite fibers. It is proper by
projectivity, and is consequently finite.

For $(P,\alpha)\in\mathfrak U$, every automorphism of $P$ lies in
$Z(G)$, since the associated adjoint bundle has no nontrivial
automorphisms. Such an automorphism preserves $\alpha$ exactly
when it lies in
\[
 Z(G)\cap G_{\mathrm{der}}=Z(G_{\mathrm{der}}).
\]
Consequently, $\mathfrak U\to U$ is a gerbe with this band.
The lifting description above shows that the restriction of
$\phi$ to $U$ is finite \etale{} over the product of the regularly
stable loci, as in the proof of \cite[Proposition~13.5]{BLS98}.
Those loci are smooth by \cite[\S13.1]{BLS98}, so $U$ is smooth.
On $\mathfrak U$, we have $H^0(C,E_{\Ad}(P))=0$.
Riemann--Roch therefore gives
\[
 \dim M=h^1(C,E_{\Ad}(P))=(g-1)\dim G_{\mathrm{der}}.
\]

For each $G_i$, the complement of the stable locus in
$\MM_{G_i}^{\tau_i}$ has codimension at least $g-1$
\cite[Lemma~2.1]{BH12}. The complement of the regularly stable
locus in $M_{G_i}^{\tau_i}(C)$ has codimension at least two
\cite[\S13.1]{BLS98}, by the argument of
\cite[Theorem~II.6]{Fal93}. The genus-two exception in that
argument does not occur here. On the stable locus, the stabilizers
are finite, so passage to the coarse moduli space preserves
codimension. Combining these two estimates gives codimension
at least two for the complement of the regularly stable locus
in each stack $\MM_{G_i}^{\tau_i}$. Smoothness of $\Phi$ then
gives the asserted bound for $\MM\setminus\mathfrak U$.
Finiteness and surjectivity of $\phi$ give the bound for
$M\setminus U$.

Since $M$ is normal, $U$ is smooth, and $M\setminus U$ has
codimension at least two, restriction identifies
$\Cl(M)$ with $\Pic(U)$; see also \cite[\S13.1]{BLS98}.
Pullback along the gerbe $\mathfrak U\to U$ is injective on
Picard groups \cite[\S7]{BH12}. Moreover,
$\Pic(\MM)=\Pic(\mathfrak U)$ by smoothness and the
codimension bound \cite[Lemma~7.3(ii)]{BH12}.
The Picard calculation for the twisted stack
\cite[Lemma~1.4, Proposition~1.5 and \S2.4]{BLS98} thus gives
\[
 \Cl(M)_{\mathbb Q}
 =\Pic(U)_{\mathbb Q}
 \lhook\joinrel\longrightarrow
 \Pic(\mathfrak U)_{\mathbb Q}
 =\Pic(\MM)_{\mathbb Q}
 \simeq\mathbb Q^s.
\]
For each $i$, the group $\Pic(M_{G_i}^{\tau_i}(C))$ is generated
by an ample line bundle \cite[Proposition~7.4]{BLS98}.
Choose $\Theta_i$ to be a positive rational multiple of the
pullback of its class under the $i$th component of $\phi$.
These classes are linearly independent over $\mathbb Q$:
a relation among their pullbacks would make the corresponding
class on the product numerically trivial, since $\phi$ is
finite and surjective. Restriction to each factor then makes
every coefficient zero.
They therefore form a basis of $\Cl(M)_{\mathbb Q}$.
They are nef, and positive multiples are nonzero effective
Cartier divisors. This proves \eqref{eq:theta-positive}.
\end{proof}

The lifting argument and irreducibility used above also hold for
$g\geq2$; the codimension-two assertions are the part for which
we impose $g\geq3$.

\begin{remark}
If $G_{\mathrm{der}}$ is almost simple, then
$M=M_{G,Q}^{\tau}(C)$ is a normal Fano variety of Picard number
one. Indeed, the proof of \cite[Proposition~13.6]{BLS98}
identifies the anticanonical line bundle of
$M_{G_1}^{\tau_1}(C)$ with the determinant line bundle associated
with the adjoint representation. This line bundle is ample by
\cite[Construction~II.8]{Fal93}, since the adjoint representation
of $G_1$ is faithful. The morphism $\phi$ is \etale{} over the
regularly stable locus, and $M\setminus U$ has codimension at
least two. Hence
\[
 K_M=\phi^*K_{M_{G_1}^{\tau_1}(C)}.
\]
Thus $-K_M$ is Cartier and ample. The Picard number assertion
follows from Proposition~\ref{prop:inputs} with $s=1$.
\end{remark}

\begin{lemma}\label{lem:adjoint-reduction}
Suppose $\rho|_{Z(G)}=1$. For every $P\in\MM_{G,Q}^{\tau}$
there is a canonical identification
\begin{equation}\label{eq:associated-adjoint}
 E_\rho(P)=P\times^G V
 \simeq(P\times^G G_{\mathrm{ad}})\times^{G_{\mathrm{ad}}}V.
\end{equation}
A generic cohomology-vanishing statement for this representation
on the induced component of $\MM_{G_{\mathrm{ad}}}\times C$
therefore holds on $\MM_{G,Q}^{\tau}\times C$ as well.
\end{lemma}

\begin{proof}
The identification follows from extension of structure group,
since $\rho$ factors through $G_{\mathrm{ad}}$. A nonempty
vanishing open pulls back to a nonempty open by surjectivity of
$\Phi$; its inverse image is dense by irreducibility.
\end{proof}

This does not identify $\MM_{G,Q}^{\tau}$ with a component of
$\MM_{G_{\mathrm{der}}}$. For example, a rank-$n$ vector bundle
with specified determinant $L$ cannot in general be converted
into an $\SL_n$-bundle by tensoring: this requires an $n$th root
of $L$. Its projective bundle nevertheless lies in the
appropriate component of $\MM_{\PGL_n}$, and its trace-free
endomorphism bundle depends only on that projective bundle.
Lemma~\ref{lem:adjoint-reduction} transfers cohomology vanishing
in this way. Stability will instead be proved directly on
$M_{G,Q}^{\tau}(C)$, using modifications in $G_{\mathrm{der}}$.

\subsection{Universal associated bundles and cohomology sheaves}

Let $\mathcal P$ be the universal principal $G$-bundle on
$\MM\times C$. It comes with the universal identification of
$\mathcal P\times^G G^{\mathrm{ab}}$ with the pullback of $Q$.
Applying the representation gives the associated vector bundle
$\mathcal P\times^G V$ on $\MM\times C$.

\begin{proposition}\label{prop:associated-descent}
The restriction of $\mathcal P\times^G V$ to
$\mathfrak U\times C$ descends to a vector bundle $\EE_\rho$ on
$U\times C$. Its restriction to $\{[P]\}\times C$ is $E_\rho(P)$, its rank is $r$,
and its determinant is trivial. In particular,
$\deg E_\rho(P)=0$.
The same assertions hold for $M=M_G^\tau(C)$, with $U$ the
inverse image of the regularly stable loci in the adjoint
factors and $\mathfrak U\subset\MM_G^\tau$ its inverse image
in the bundle stack.
\end{proposition}

\begin{proof}
For $M=M_{G,Q}^{\tau}(C)$, every stabilizer of
$\mathfrak U\to U$ is contained in $Z(G)$ and therefore acts
trivially on $V$. Descent along this gerbe gives $\EE_\rho$.
For $M=M_G^\tau(C)$, the morphism $\mathfrak U\to U$ is
instead a gerbe with band $Z(G)$ \cite[\S6]{BH12}.
The same triviality of the stabilizer action gives the descent
of the vector bundle associated with the universal $G$-bundle
on $\MM_G^\tau\times C$.
In either case, $\det\rho$ is a character of the connected
semisimple group $G_{\mathrm{ad}}$, which has no nontrivial
characters. Thus $\det\rho=1$, and the associated determinant
bundle is trivial.
\end{proof}

For the following construction, allow either choice of $M$ and
retain $g\geq3$.
The space $M$ is normal, $U$ is smooth, and
\[
 \codim_M(M\setminus U)\geq2.
\]
For $M=M_{G,Q}^{\tau}(C)$, these assertions are part of
Proposition~\ref{prop:inputs}. For $M=M_G^\tau(C)$, they follow
from the finite \etale{} cover constructed in the proof of
Proposition~\ref{prop:varying-central-data}. The construction of
this cover and these assertions about $M$ and $U$ do not require
a bound on $m_\rho$.

Let $p:U\times C\to U$ be the projection. Since
$(R^1p_*\EE_\rho)^{**}$ is reflexive on $U$, it has a unique
coherent reflexive extension to $M$
\cite[Tags~0AY4 and~0EBJ]{Stacks}. This is the sheaf
$\mathcal V_\rho$ defined in \eqref{eq:Vr-intro}.

Set
\begin{equation}\label{eq:Urho}
 U_\rho=\{[P]\in U:H^0(C,E_\rho(P))=0\},\qquad
 j_\rho:U_\rho\hookrightarrow M.
\end{equation}
Semicontinuity makes $U_\rho$ open, and its nonemptiness will
follow from Theorem~\ref{thm:interp}. On this open set,
cohomology and base change makes $(R^1p_*\EE_\rho)|_{U_\rho}$
locally free of rank $(g-1)r$. Taking its double dual does not
change it, so there is a canonical identification
\[
 \mathcal V_\rho|_{U_\rho}
 \simeq\mathcal V_\rho^\circ
 :=(R^1p_*\EE_\rho)|_{U_\rho}.
\]
In particular, $\rk\mathcal V_\rho=(g-1)r$.
For a curve $f:\Pp^1\to U_\rho$, we may therefore identify
$f^*\mathcal V_\rho$ with $f^*\mathcal V_\rho^\circ$.
Under the additional bound $m_\rho\leq g-1$,
Section~\ref{subsec:cohomology-extension} will prove
$\codim_M(M\setminus U_\rho)\geq2$ and the resulting
identification $\mathcal V_\rho\simeq(j_\rho)_*\mathcal V_\rho^\circ$.

For either choice of $M$, a torsion-free coherent sheaf
$\mathcal A$ on $M$ and an ample Cartier divisor $H$, put
\[
 \det\mathcal A=
 \bigl(\textstyle\bigwedge^{\rk\mathcal A}\mathcal A\bigr)^{**},
 \qquad
 \mu_H(\mathcal A)=
 \frac{c_1(\mathcal A)\cdot H^{\dim M-1}}{\rk\mathcal A}.
\]
Here $c_1(\mathcal A)$ denotes the Weil divisor class of
$\det\mathcal A$, and intersection with the Cartier divisor
$H$ is understood in the Chow group
\cite[Tags~02SN and~02TG]{Stacks}.
For $M=M_{G,Q}^{\tau}(C)$, Proposition~\ref{prop:inputs}
also shows that this class is rational Cartier.
The sheaf $\mathcal A$
is slope stable with respect to $H$ if
$\mu_H(\mathcal F)<\mu_H(\mathcal A)$ for every nonzero
subsheaf $\mathcal F\subset\mathcal A$ of smaller rank.
It suffices to test saturated subsheaves. On a smooth open
set, a torsion-free coherent sheaf is locally free away from
a subset of codimension at least two.

\section{The vanishing theorem and vanishing sequences}
\label{sec:interpolation}

In this section, we prove the vanishing theorem stated in
Theorem~\ref{thm:interp}. We first treat the case of a connected semisimple
group and then deduce the reductive case by passing to the
adjoint quotient, using Lemma~\ref{lem:adjoint-reduction} for the assertion with fixed
central data. We conclude by proving the monotonicity of
vanishing increments in Lemma~\ref{lem:jets}, which will be used in
Section~\ref{sec:slope}.

For the preliminary lemmas below, let $G$ be connected semisimple,
let $g\geq 2$, and let $(\rho,V)$ be a finite-dimensional algebraic
representation satisfying $V^G=0$; irreducibility is not required.
Write $\mathfrak{g}=\operatorname{Lie}(G)$ and retain the notation
$E_\rho(P)=P\times^G V$.

\subsection{Finite monodromy and invariant vectors}
\begin{lemma}\label{lem:finite-monodromy}
For every $\tau\in\pi_1(G)$, there is a bundle
$P_0\in\MM_G^\tau$ induced from a connected finite \'etale
$A$-torsor, where $A\subset G$ is finite abelian and contains a
regular semisimple element.
\end{lemma}

\begin{proof}
Let
$1\longrightarrow Z\longrightarrow\widetilde G
 \xrightarrow{q}G\longrightarrow1$
be the simply connected covering of $G$. Its kernel $Z$
is finite and central, and is naturally identified with
$\pi_1(G)$. Let $\zeta\in Z$ correspond to the prescribed
topological type $\tau$.

Choose a maximal torus $T'\subset\widetilde G$ and a
Borel subgroup $B'\supset T'$. Let
$\alpha_1,\ldots,\alpha_{\operatorname{rank}\widetilde G}$
be the corresponding simple roots.
For each $i$, let $U_{\alpha_i}$ be the root subgroup,
choose an isomorphism of algebraic groups
 $x_i:\mathbb G_a\xrightarrow{\sim}U_{\alpha_i},$
and choose an element $\dot s_i\in N_{\widetilde G}(T')$
representing the simple reflection $s_i$ associated
with $\alpha_i$. Define $n=\dot s_1\cdots\dot s_{\operatorname{rank}\widetilde G},\ w=s_1\cdots s_{\operatorname{rank}\widetilde G}.$ Thus $n$ represents $w$ in
$N_{\widetilde G}(T')/T'$, and the action of $w$
on $T'$ is $w(t)=ntn^{-1}.$ Consider the morphism
$\phi: T'\longrightarrow T'$ which sends $t$ to $w(t)t^{-1}=ntn^{-1}t^{-1}.$ Then by \cite[Lemma~7.6]{Ste65}, the map $\phi$ is an isogeny, namely a surjective
homomorphism of algebraic groups with finite kernel.
In particular, the group
$\ker\phi=(T')^w:=\{t\in T':w(t)=t\}$ is a finite group.
 Since the central subgroup $Z$ is contained in $T'$,
surjectivity of $\phi$ therefore allows us to choose $t_0\in T'$
such that $\phi(t_0)=[n,t_0]=\zeta.$

We claim that both $n$ and $t_0$ have finite order. Indeed, let $h$
be the order of $w$ in the Weyl group. Then $n^h\in T'$,
and $w(n^h)=n\,n^h\,n^{-1}=n^h.$ Hence $n^h\in(T')^w$. Since $(T')^w$ is finite, $n^h$, and consequently $n$, has finite order. Similarly, if $d$ is the order of $\zeta$, then
$\phi(t_0^d)=\phi(t_0)^d=\zeta^d=1.$ Thus $t_0^d\in(T')^w$, so $t_0$ has finite order.

Put $a=q(n),\ b=q(t_0),\ A=\langle a,b\rangle.$
The elements   $a$ and $b$ commute since $[a,b]=q([n,t_0])=q(\zeta)=1.$ They both have finite order,
so $A$ is a finite abelian subgroup of $G$. We claim that $a$ is a regular semisimple element. By Steinberg's cross-section theorem \cite[Theorem~1.4]{Ste65}, $n$ is regular, meaning that $\dim C_{\widetilde G}(n)=\operatorname{rank}\widetilde G.$
Moreover, $n$ is semisimple because it has finite order. Hence $n$ is  regular semisimple. This implies that $a=q(n)$ is regular semisimple since $q$ is a central isogeny.

Fix a base point $*\in C$. Then we have
$\pi_1(C,*)
 =
 \left\langle
 \gamma_1,\delta_1,\ldots,\gamma_g,\delta_g
 \ \middle|\
 \prod_{j=1}^{g}[\gamma_j,\delta_j]=1
 \right\rangle.$ Define   a surjective group morphism $\theta:\pi_1(C,*)\longrightarrow A$ by
\[
 \begin{aligned}
 \theta(\gamma_1)&=a,
 &\theta(\delta_1)&=b,\\
 \theta(\gamma_j)&=1,
 &\theta(\delta_j)&=1
 \qquad(2\leq j\leq g).
 \end{aligned}
\]

The homomorphism $\theta$ defines a connected finite topological
$A$-covering of $C^{\mathrm{an}}$. By the Riemann existence theorem,
it is the analytification of a connected finite \etale{}
$A$-torsor $E\to C$.

Extend its structure group along $A\subset G$ and set
$P_0=E\times^A G$. The product of commutators of the chosen lifts
of its monodromy generators is $[n,t_0]=\zeta$, so its topological
type is $\tau$.

\end{proof}

Choose $A$ and $P_0$ as in Lemma~\ref{lem:finite-monodromy}.
The character decomposition gives
\begin{equation}\label{eq:finite}
 E_\rho(P_0)=(V^A\otimes\OO_C)
 \oplus\bigoplus_{\chi\ne1}V_\chi\otimes\eta_\chi,
\end{equation}
where the $\eta_\chi$ are nontrivial degree-zero torsion line bundles.
The $A$-equivariant projection
$\Pi=\frac1{|A|}\sum_{a\in A}\rho(a):V\longrightarrow V^A$
induces the projection onto the first summand of \eqref{eq:finite}.

\begin{lemma}\label{lem:compression}
Notation as above. There is $X\in\gfr$ such that
\begin{equation}\label{eq:compression}
 \Pi\,d\rho(X)^2|_{V^A}:V^A\longrightarrow V^A
\end{equation}
is an isomorphism.
\end{lemma}

\begin{proof}
If $V^A=0$, there is nothing to prove. Assume $V^A\ne0$.
Choose a regular semisimple element $a\in A$, and set
 $T=C_G(a)^\circ,\ \hfr=\Lie T.$
Then $T$ is a maximal torus containing $a$. Since $A$ is
abelian, it centralizes $a$ and therefore normalizes $T$.
Let $\Phi$ be the root system of $(G,T)$, and write
$V=\bigoplus_\lambda V_\lambda,\ V_0=V^T,\ V_{\ne0}=\bigoplus_{\lambda\ne0}V_\lambda.$ For a character $\lambda$ of $T$ and $H\in\hfr$, we write
$\lambda(H)$ for the value of the differential of $\lambda$
at $H$. Because $A$ normalizes $T$, both $V_0$ and
$V_{\ne0}$ are $A$-stable. Hence $V^A=(V^A\cap V_0)\oplus(V^A\cap V_{\ne0}).$
The natural projection $q_0:V\to V_0$  commutes
with $A$, and therefore with $\Pi$.

\smallskip\noindent\textbf{Claim.}
There exist $X_0, Y_0\in \gfr$ such that relative to  the decomposition $V^A=(V^A\cap V_0)\oplus(V^A\cap V_{\ne0})$, we have
\[
 \Pi d\rho(X_0)^2|_{V^A}
 =
 \begin{pmatrix}
  0&0\\
  0&B
 \end{pmatrix},
 \qquad
 \Pi d\rho(Y_0)^2|_{V^A}
 =
 \begin{pmatrix}
  C_{11}&C_{12}\\
  C_{21}&C_{22}
 \end{pmatrix},
\]
with $B$ and $C_{11}$ invertible, and $$\Pi\bigl(
 d\rho(X_0)d\rho(Y_0)+d\rho(Y_0)d\rho(X_0)
 \bigr)|_{V^A}=0.$$

Assume the claim holds. If $V^A\cap V_0=0$, take $X=X_0$.
If $V^A\cap V_{\ne0}=0$, take $X=Y_0$.
In the remaining case,  we put $Q(t)=\Pi d\rho(X_0+tY_0)^2|_{V^A}.$ We show that there is a $t_0\in \C$ such that $Q(t_0)$ is an isomorphism.
By the choice of $X_0$ and $Y_0$, we have
\[
 Q(t)
 =
 \begin{pmatrix}
  t^2C_{11}&t^2C_{12}\\
  t^2C_{21}&B+t^2C_{22}
 \end{pmatrix}.
\]
Put $m=\dim(V^A\cap V_0)$. Factoring $t^2$ from
each of the first $m$ rows gives
\[
 \det Q(t)
 =
 t^{2m}
 \det
 \begin{pmatrix}
  C_{11}&C_{12}\\
  t^2C_{21}&B+t^2C_{22}
 \end{pmatrix}=t^{2m}f(t),
\]
where $f(t)$ is a polynomial with $f(0)=\det
 \begin{pmatrix}
  C_{11}&C_{12}\\
  0&B
 \end{pmatrix}
 =\det(C_{11})\det(B)\ne0.$ Hence $\det Q(t)$ is not the zero polynomial. Choose $t_0\in\C$ outside its finite zero set and set
$X=X_0+t_0Y_0$. Then
\[
 \Pi d\rho(X)^2|_{V^A}:V^A\longrightarrow V^A
\]
is an isomorphism, as required.

It remains to prove the claim. We construct $X_0$ and $Y_0$.

\textbf{Construction of $X_0$:}
Choose a cocharacter $\mu:\mathbb{G}_m\to T$ such that
$\langle\lambda,\mu\rangle\ne0$
for every nonzero weight $\lambda$ occurring in $V$, and
put $X_0=d\mu(1)\in\hfr$. Such a cocharacter exists because
the cocharacter lattice is not contained in a finite union
of proper hyperplanes. We will show that relative to  the decomposition $V^A=(V^A\cap V_0)\oplus(V^A\cap V_{\ne0})$, we have  $\Pi d\rho(X_0)^2|_{V^A}
=
\begin{pmatrix}
	0&0\\
	0&B
\end{pmatrix}$ with $B$ invertible.
Consider the operator
\[
 \mathcal B
 =
 \frac1{|A|}\sum_{b\in A}
 \rho(b)d\rho(X_0)^2\rho(b)^{-1}
 =
 \frac1{|A|}\sum_{b\in A}
 d\rho(\Ad(b)X_0)^2.
\]
This operator commutes with $A$. Moreover, every
$\Ad(b)X_0$ belongs to $\hfr$, since $A$ normalizes $T$.
Thus $\mathcal B$ preserves each weight space $V_\lambda$
and acts there by the scalar
\[
 \beta_\lambda
 =
 \frac1{|A|}\sum_{b\in A}
 \lambda(\Ad(b)X_0)^2.
\]
Each number $\lambda(\Ad(b)X_0)$ is an integer: it is
the pairing of $\lambda$ with the cocharacter
$b\mu b^{-1}$ of $T$. Consequently, $\beta_\lambda>0$
for every nonzero weight $\lambda$ occurring in $V$,
because all summands are nonnegative and the summand
corresponding to $b=1$ is nonzero. Also $\beta_0=0$. For $v\in V^A$, we have $\rho(b)^{-1}v=v$, so $\mathcal Bv=\Pi d\rho(X_0)^2v.$ It follows that $ \Pi d\rho(X_0)^2|_{V^A}=\begin{pmatrix}
  0&0\\
  0&B
 \end{pmatrix}$ with $B$ invertible.
\\

\textbf{Construction of $Y_0$:}  Let $\kappa(x,y)=\tr_{\gfr}\bigl(\ad(x)\ad(y)\bigr)$ be the  Killing form. Fix a positive system $\Phi^+\subset\Phi$. For every
$\alpha\in\Phi^+$, choose a nonzero
$e_\alpha\in\gfr_\alpha$ and choose
$e_{-\alpha}\in\gfr_{-\alpha}$ such that
$
\kappa(e_\alpha,e_{-\alpha})=1.
$ Choose $\kappa$-dual bases $(h_j)$ and $(h^j)$ of
$\hfr$, so that $\kappa(h_j,h^k)=\delta_j^k.$

Set $Y_0=\sum_{\alpha\in\Phi}e_\alpha$. Since $q_0$ commutes
with $\Pi$, the block $C_{11}$ in the claim is
$\Pi q_0d\rho(Y_0)^2|_{V^A\cap V_0}$.

Consider the Casimir operator $\Omega_V=\sum_j d\rho(h_j)d\rho(h^j) + \sum_{\alpha\in\Phi}
 d\rho(e_\alpha)d\rho(e_{-\alpha})$. Then $\Omega_V:V\rightarrow V$ is an isomorphism. In fact, by the Casimir eigenvalue formula \cite[\S22.3, p.~122]{Hum72}, the Casimir element associated with the Killing form acts by a nonzero scalar on every nontrivial finite-dimensional irreducible $\gfr$-module, and by zero on the trivial module. Complete reducibility
\cite[\S6.3]{Hum72} therefore gives $\ker\Omega_V=V^{\gfr}.$ Since $G$ is connected and $V^G=0$, we have $V^{\gfr}=V^G=0$. Hence $\Omega_V$ is an isomorphism.  Since $\Omega_V$ commutes with both $A$ and $T$,
it preserves $V^A\cap V_0$. Its restriction $\Omega_V|_{V^A\cap V_0}:V^A\cap V_0\rightarrow V^A\cap V_0$ is therefore invertible. Furthermore, we have $q_0d\rho(Y_0)^2|_{V^A\cap V_0}=\Omega_V|_{V^A\cap V_0}.$ In fact, for $v\in V_0$, we have $d\rho(e_\alpha)d\rho(e_\beta)v\in V_{\alpha+\beta},$
so only the terms with $\beta=-\alpha$ contribute
to the zero-weight component of $d\rho(Y_0)^2v$.
Consequently,
\[
 q_0d\rho(Y_0)^2v
 =
 \sum_{\alpha\in\Phi}
 d\rho(e_\alpha)d\rho(e_{-\alpha})v
 =
 \Omega_Vv.
\]
Since $\Omega_V$ preserves $V^A\cap V_0$, applying $\Pi$ gives
$C_{11}=\Omega_V|_{V^A\cap V_0}$, which is invertible.

Finally we show the   equality $\Pi\bigl(d\rho(X_0)d\rho(Y_0)+d\rho(Y_0)d\rho(X_0)
\bigr)|_{V^A}=0.$ Since $a$ is regular semisimple, we have $\alpha(a)\ne1$ for all $\alpha\in\Phi$. Moreover, the equality $\Pi\rho(a)=\Pi$ implies that $\Pi$
annihilates every eigenspace of $\rho(a)$ with
eigenvalue different from $1$. Now let $v\in V^A$. Then
$\rho(a)v=v$. Also $\rho(a)$ commutes with $d\rho(X_0)$,
while
$\rho(a)d\rho(e_\alpha)\rho(a)^{-1} =\alpha(a)d\rho(e_\alpha).
$
It follows that both
$d\rho(X_0)d\rho(e_\alpha)v
 \ \text{and}\
 d\rho(e_\alpha)d\rho(X_0)v$
belong to the $\alpha(a)$-eigenspace of $\rho(a)$,
and hence are killed by $\Pi$. Summing over
$\alpha\in\Phi$ gives $\Pi\bigl(d\rho(X_0)d\rho(Y_0)+d\rho(Y_0)d\rho(X_0)
\bigr)|_{V^A}=0.$
\end{proof}

\subsection{Vanishing sequences and principal parts}
\label{subsec:vanishing-principal-parts}

Throughout this subsection, $C$ is a smooth connected projective
curve over $\C$. We collect some facts about vanishing sequences and
principal parts.

\smallskip\noindent
\textbf{Vanishing sequences.}
Let $M$ be a line bundle on $C$, and let
$S\subset H^0(C,M)$ be a subspace of dimension $N>0$. For $x\in C$ and
$m\geq0$, put
\[
 S(-mx):=S\cap H^0(C,M(-mx)).
\]

\begin{definition}
   The \emph{vanishing sequence} of $S$ at $x$ is the increasing
sequence
\[
 0\leq\nu_0(S,x)<\cdots<\nu_{N-1}(S,x)
\]
of vanishing orders of the nonzero sections of $S$ at $x$.
\end{definition}
 Equivalently, these are the integers
$m$ for which
\[
 \dim S(-mx)-\dim S(-(m+1)x)=1.
\]
Indeed, each successive quotient injects into the one-dimensional
fiber $M(-mx)|_x$, and the filtration eventually becomes zero.
In particular,
\begin{equation}\label{eq:vanishing-sequence-count}
 \dim S(-mx)=\#\{i:\nu_i(S,x)\geq m\}.
\end{equation}

\begin{proposition}\cite[Corollary 3.2]{DP24}
\label{prop:generic-vanishing-sequences}
Let $M$ be a line bundle on $C$, and let
$S\subset H^0(C,M)$ have dimension $N>0$. There is a nonempty
Zariski open subset $U\subset C$ such that, for every $x\in U$,
\[
 (\nu_0(S,x),\ldots,\nu_{N-1}(S,x))=(0,1,\ldots,N-1).
\]
Consequently,
\begin{equation}\label{eq:generic-vanishing-count}
 \dim S(-mx)=\max\{N-m,0\}\qquad(m\geq0).
\end{equation}
\end{proposition}

\smallskip\noindent
\textbf{Principal parts.}
Let $F$ be a vector bundle on $C$, fix $x\in C$, and let $n\geq1$.
\begin{definition}
    The space of $F$-valued principal parts at $x$ of order
at most $n$ is the space \[
 \operatorname{PP}^{\leq n}_x(F)
 :=H^0\bigl(C,F(nx)/F\bigr).
\]
\end{definition}

 Choose a uniformizer $z$ at $x$, and write
$\widehat F_x$ for the space $H^0(D_x,F)$, where $D_x=\operatorname{Spec}\C[[z]]$ is the formal disc. Then
\[
 \operatorname{PP}^{\leq n}_x(F)
 \simeq z^{-n}\widehat F_x/\widehat F_x.
\]
For $\xi\in z^{-n}\widehat F_x$, we define the principal part of $\xi$  by the class
\[
 \operatorname{pp}_x(\xi)
 :=\xi\bmod\widehat F_x.
\]
Locally this is represented by the negative-power terms
of the Laurent expansion of $\xi$. The quotient and its class
are intrinsic, even though this representative depends on the
coordinate and frame.

The exact sequence
\begin{equation}\label{eq:principal-parts-exact-sequence}
 0\longrightarrow F\longrightarrow F(nx)
 \longrightarrow F(nx)/F\longrightarrow0
\end{equation}
gives a connecting homomorphism
\begin{equation}\label{eq:principal-parts-boundary}
 \partial^F_{x,n}:\operatorname{PP}^{\leq n}_x(F)
 \longrightarrow H^1(C,F).
\end{equation}
We write
\begin{equation}\label{eq:principal-parts-class-notation}
 [\xi]:=\partial^F_{x,n}\bigl(\operatorname{pp}_x(\xi)\bigr).
\end{equation}
The connecting maps are compatible with increasing $n$, so this
class is independent of the chosen pole bound.

\begin{proposition}
\label{prop:principal-parts-duality}
Let $F$ be a vector bundle on $C$, let $x\in C$, and let
$\xi\in z^{-n}\widehat F_x$ for some $n\geq1$.
\begin{enumerate}
\item\label{item:principal-parts-lift}
The class $[\xi]$ is zero if and only if there is a section
$t\in H^0(C,F(nx))$ satisfying
\[
 \operatorname{pp}_x(t)=\operatorname{pp}_x(\xi).
\]

\item\label{item:principal-parts-residue}
For every $\omega\in H^0(C,F^\vee\otimes K_C)$, the Serre
pairing is given by
\begin{equation}\label{eq:principal-parts-residue-pairing}
 \bigl\langle[\xi],\omega\bigr\rangle
 =\operatorname{Res}_x\bigl(\omega(\xi)\bigr).
\end{equation}

\item\label{item:principal-parts-naturality}
For every morphism of vector bundles $\psi:F\to F'$ on $C$,
\begin{equation}\label{eq:principal-parts-naturality}
 H^1(C,\psi)([\xi])=[\psi(\xi)]\in H^1(C,F').
\end{equation}
\end{enumerate}
\end{proposition}

\begin{proof}
The first assertion follows from exactness of
\[
 H^0(C,F(nx))\longrightarrow
 \operatorname{PP}^{\leq n}_x(F)
 \xrightarrow{\partial^F_{x,n}}H^1(C,F).
\]
The third follows from naturality of connecting homomorphisms.

For the second assertion, regard $\omega$ as a morphism
$F\to K_C$. Naturality gives
\[
 H^1(C,\omega)([\xi])=[\omega(\xi)]\in H^1(C,K_C).
\]
The Serre pairing is the trace of this class, and the trace
of a differential principal part supported at $x$ is its
residue there; see \cite[\S\S17--18]{For81}.
This proves \eqref{eq:principal-parts-residue-pairing}.
\end{proof}

Set $C^\circ=C\setminus\{x\}$ and
$D_x^\times=\operatorname{Spec}\C((z))$, with $D_x$ as above.
Let $\varphi$ denote restriction from $C^\circ$ to the punctured
disc, expressed in a chosen frame on $D_x$, and define
\begin{equation}\label{eq:principal-parts-matching-map}
 \begin{aligned}
 \delta_{F,x}:\;&\Gamma(C^\circ,F)\oplus\Gamma(D_x,F)
 \longrightarrow\Gamma(D_x^\times,F),\\
 &\delta_{F,x}(a,b)=b-\varphi(a).
 \end{aligned}
\end{equation}

\begin{proposition}
\label{prop:principal-parts-matching}
With the notation above, there is an exact sequence
\[
 0\longrightarrow H^0(C,F)
 \longrightarrow\Gamma(C^\circ,F)\oplus\Gamma(D_x,F)
 \xrightarrow{\delta_{F,x}}\Gamma(D_x^\times,F)
 \xrightarrow{\ \xi\mapsto[\xi]\ }H^1(C,F)
 \longrightarrow0.
\]
The first map sends a global section to its two restrictions.
In particular,
\begin{equation}\label{eq:principal-parts-matching-cohomology}
 \ker\delta_{F,x}=H^0(C,F),\qquad
 \operatorname{coker}\delta_{F,x}\simeq H^1(C,F),
\end{equation}
and, for $\xi\in\Gamma(D_x^\times,F)$,
\begin{equation}\label{eq:principal-parts-matching-solvability}
 \delta_{F,x}(a,b)=\xi\text{ has a solution}
 \quad\Longleftrightarrow\quad [\xi]=0.
\end{equation}
\end{proposition}

\begin{proof}
The cohomology sequence of
\eqref{eq:principal-parts-exact-sequence} is
\[
 \begin{aligned}
 0\longrightarrow H^0(C,F)\longrightarrow H^0(C,F(nx))
 &\longrightarrow\operatorname{PP}^{\leq n}_x(F)\\
 &\longrightarrow H^1(C,F)\longrightarrow H^1(C,F(nx))
 \longrightarrow0.
 \end{aligned}
\]
The line bundle $\OO_C(x)$ is ample, so Serre vanishing gives
$H^1(C,F(nx))=0$ for $n\gg0$. Moreover,
\[
 \begin{aligned}
 \bigcup_{n\geq0}H^0(C,F(nx))&=\Gamma(C^\circ,F),\\
 \varinjlim_n\operatorname{PP}^{\leq n}_x(F)
 &=\widehat F_x[z^{-1}]/\widehat F_x
 =\Gamma(D_x^\times,F)/\Gamma(D_x,F).
 \end{aligned}
\]
Taking the direct limit therefore gives an exact sequence
\[
 \begin{aligned}
 0\longrightarrow H^0(C,F)\longrightarrow\Gamma(C^\circ,F)
 &\longrightarrow\Gamma(D_x^\times,F)/\Gamma(D_x,F)\\
 &\longrightarrow H^1(C,F)\longrightarrow0.
 \end{aligned}
\]
The last nonzero map sends the principal part of $\xi$ to $[\xi]$.
Quotienting $\Gamma(D_x^\times,F)$ first by the regular disc
sections and then by the image of $\Gamma(C^\circ,F)$ gives
$\operatorname{coker}\delta_{F,x}$. The minus sign in
$\delta_{F,x}$ does not change this image. Its kernel consists of
pairs of matching restrictions, hence of global sections of $F$. This proves the asserted exact sequence and the
solvability criterion.
\end{proof}

\subsection{The vanishing theorem}
We use the following deformation lemma to prove
Theorem~\ref{thm:interp}.

\begin{lemma}
\label{lem:one-point-gluing-deformation}
Fix a $G$-bundle $P_0$, a point $p\in C$, a uniformizer $z$ at $p$,
and a trivialization of $P_0$ on the formal disc at $p$. Given
\[
 u=u(z)\in\C[z,z^{-1}],\qquad X\in\gfr,
\]
there is an algebraic family $P_\eps$ of $G$-bundles parametrized by
$\mathbb A^1=\operatorname{Spec}(\C[\eps])$, with special fiber $P_0$,
obtained by multiplying the original gluing from $C\setminus\{p\}$
to the disc on the left by
$\gamma(\eps,z)\in G(\C[\eps]((z)))$.
The family stays in the component of $P_0$, and
\begin{equation}\label{eq:secondjet}
 \rho\bigl(\gamma(\eps,z)\bigr)
 =I_V+\eps uA_X+\frac{\eps^2u^2}{2}A_X^2
 \pmod{\eps^3},
\end{equation}
where $A_X=d\rho(X)$. The gluing of $E_\rho(P_\eps)$ is obtained
by multiplying that of $E_\rho(P_0)$ on the left by
$\rho(\gamma(\eps,z))$.

\end{lemma}

\begin{proof}
Nilpotent elements span the semisimple Lie algebra $\gfr$: their span is
invariant under the adjoint action and is nonzero on every simple ideal.
Write $X=X_1+\cdots+X_t$ with each $X_i$ nilpotent, and put
\begin{equation}\label{eq:palindrome}
 \gamma(\eps,z)=
 \left(\prod_{i=1}^t\exp(\eps u(z)X_i/2)\right)
 \left(\prod_{i=t}^1\exp(\eps u(z)X_i/2)\right).
\end{equation}

Each nilpotent $X_i$ defines an algebraic one-parameter unipotent
subgroup, so $\gamma(\eps,z)\in G(\C[\eps]((z)))$.
The coefficient of $\eps$ in $\rho(\gamma)$ is $uA_X$.
The reversed order of the second product gives
$\gamma(-\eps,z)=\gamma(\eps,z)^{-1}$.
Comparing the coefficients of $\eps^2$ in
$\rho(\gamma(\eps,z))\rho(\gamma(-\eps,z))=I_V$ shows that
the quadratic coefficient in $\rho(\gamma)$ is $u^2A_X^2/2$.
This proves \eqref{eq:secondjet}.

Apply Beauville--Laszlo gluing~\cite{BL95} over $\C[\eps]$, using the
original bundles on $C\setminus\{p\}$ and on the formal disc, and the
modified identification on the punctured disc. Gluing is compatible with
tensor products, so the $G$-structure glues as well. Since
$\gamma(0,z)=1$, the special fiber is $P_0$. The parameter space is
connected, so the topological type is unchanged.
\end{proof}

\begin{theorem}\label{thm:vanishing}
Let $g\geq2$, let $G$ be connected reductive, and let $(\rho,V)$ be a finite-dimensional algebraic representation satisfying $V^G=0$ and $ \rho|_{Z(G)}=1$. For every $\tau$ and $Q$ as above, a general pair $(P,x)$ in either $\MM_G^\tau\times C$ or $\MM_{G,Q}^{\tau}\times C$ satisfies
\begin{equation}\label{eq:vanishing}
 H^0\bigl(C,E_\rho(P)((g-1)x)\bigr)  =H^1\bigl(C,E_\rho(P)((g-1)x)\bigr)=0.
\end{equation}
\end{theorem}

\begin{proof}
We first treat the case where $G$ is connected semisimple.
For the fixed representation $(\rho,V)$ with $V^G=0$, we prove
\begin{equation}\label{eq:interp-canonical-vanishing}
 H^0\bigl(C,E_\rho(P)\otimes K_C(-(g-1)x)\bigr)=0
\end{equation}
on a nonempty open subset of the prescribed component
$\MM_G^\tau\times C$.  Choose $P_0$ and $A$ as in Lemma~\ref{lem:finite-monodromy},
and let $\Pi:V\to V^A$ be the projection. We will deform $P_0$
so that \eqref{eq:interp-canonical-vanishing} holds. By \eqref{eq:finite},
\[
 E_\rho(P_0)
 =
 (V^A\otimes\OO_C)
 \oplus
 \bigoplus_{\chi\ne1}V_\chi\otimes\eta_\chi,
\]
where the $\eta_\chi$ are nontrivial degree-zero torsion
line bundles. We have $h^0(C,K_C)=g$. For $\chi\ne1$,
Riemann--Roch and the vanishing
$h^0(C,\eta_\chi^{-1})=0$ give
$h^0(C,K_C\otimes\eta_\chi)=g-1.$ Apply Proposition~\ref{prop:generic-vanishing-sequences}
to the complete linear systems of $K_C$ and of these
finitely many $K_C\otimes\eta_\chi$. We can choose a
single point $x\in C$ such that
\begin{equation}
   h^0\bigl(C,K_C(-(g-1)x)\bigr)=1,\quad
 h^0\bigl(C,K_C(-gx)\bigr)=h^0\bigl(C,K_C\otimes\eta_\chi(-(g-1)x)\bigr)=0.
\end{equation}

Put $L=K_C(-(g-1)x)$. Since $\deg L=g-1$,
Riemann--Roch now gives
\begin{equation}\label{eq:Lchoice}
 h^0(C,L)=h^1(C,L)=1,\quad
 h^0(C,L\otimes\eta_\chi)=
 h^1(C,L\otimes\eta_\chi)=0.
\end{equation}
Choose a generator $s\in H^0(C,L)$. Then
\begin{equation}\label{eq:interp-special-cohomology}
 \begin{aligned}
 H^0(C,E_\rho(P_0)\otimes L)&=sV^A,\\
 H^1(C,E_\rho(P_0)\otimes L)
 &\simeq V^A\otimes H^1(C,L).
 \end{aligned}
\end{equation}

If $V^A=0$, the desired vanishing already holds at
$(P_0,x)$, and semicontinuity gives the required nonempty
open subset, so we assume $V^A\ne0$. The equality
$H^0(C,L(-x))=H^0(C,K_C(-gx))=0$ shows that $s$ does not vanish at $x$ as a section of $L$. Its zero divisor has degree $g-1>0$, so it has a zero $p\ne x$. Let $a=\operatorname{ord}_p(s)\geq1.$ The line bundle $K_C\otimes L^{-1}$ is $\OO_C((g-1)x)$, and its space of sections is one-dimensional by Serre duality. Choose its canonical nonzero section $s'$, whose divisor is $(g-1)x$.
In a local parameter $z$ at $p$, write
\[
 ss'=z^a b(z)\,dz,\qquad
 b(z)=\sum_{j\geq0}b_jz^j,\qquad b_0\ne0.
\]
We choose a  Laurent polynomial $u=u(z)$ such that
\begin{equation}\label{eq:u}
 su\text{ is regular at }p,\qquad
 \operatorname{Res}_p(ss'u^2)\ne0.
\end{equation}
For $a=1$, take $u=z^{-1}$; the residue is $b_0$.
For $a\geq2$, take $u=z^{-a}+cz^{-1}$. Then $su$ is
regular at $p$ and $\operatorname{Res}_p(ss'u^2)=b_{a-1}+2cb_0,$ which is nonzero for a suitable $c\in\C$. The section $su^2$ is defined only on the punctured
formal disc at $p$ and has a pole of order at most $a$.
Using the notation of
Subsection~\ref{subsec:vanishing-principal-parts}, define
\[
 [su^2]
 :=
 \partial^L_{p,a}
 \bigl(\operatorname{pp}_p(su^2)\bigr)
 \in H^1(C,L).
\]
By
Proposition~\ref{prop:principal-parts-duality}\textup{(ii)},
\begin{equation}\label{eq:interp-nonzero-principal-class}
 \langle[su^2],s'\rangle
 =
 \operatorname{Res}_p(ss'u^2)\ne0.
\end{equation}
Hence $[su^2]\ne0$, whereas $[su]=0$ because $su$ is
regular at $p$.

Set $C^\circ=C\setminus\{p\},\
 D_p=\operatorname{Spec}\C[[z]],\
 D_p^\times=\operatorname{Spec}\C((z))$.
Trivialize on $D_p$ the finite \etale{} $A$-torsor inducing $P_0$
and use the induced framing of $P_0$; also fix a frame of $L$ on
$D_p$. In these frames, the projection in \eqref{eq:finite} is
represented by the constant operator $\Pi$, and the sections
in $sV^A$ have local expressions $sv$.
Choose $X\in\gfr$ as in Lemma~\ref{lem:compression}
and put $A_X=d\rho(X)$. Thus
\[
 \Pi A_X^2|_{V^A}:V^A\longrightarrow V^A
\]
is an isomorphism. Apply Lemma~\ref{lem:one-point-gluing-deformation}
with the chosen $u$ and $X$ at $p$. We obtain a family $P_\eps$
with special fiber $P_0$, staying in $\MM_G^\tau$, whose modifying
matrix satisfies
\[
 \rho(\gamma(\eps,z))
 =
 I_V+\eps uA_X+\frac{\eps^2u^2}{2}A_X^2
 \pmod{\eps^3}.
\]

\smallskip\noindent\textbf{Claim.}
No nonzero section
$sv\in H^0(C,E_\rho(P_0)\otimes L)=sV^A$ lifts to a section of
$E_\rho(P_\eps)\otimes L$ modulo $\eps^3$.

Assume the \textbf{claim} holds. Restrict the family to $R=\C[\eps]_{(\eps)}$ and put $\mathcal H=H^0(C_R,E_\rho(P_R)\otimes L)$. Then $\mathcal H$ is a finite $R$-module.
Multiplication by $\eps$ is injective on the locally
free family $E_\rho(P_R)\otimes L$.
The resulting long exact sequence of cohomology gives an injection
\[
 \mathcal H/\eps\mathcal H
 \lhook\joinrel\longrightarrow
 H^0(C,E_\rho(P_0)\otimes L).
\]
If $\mathcal H\ne0$, Nakayama's lemma gives a nonzero
class on the left. Its image is a nonzero special-fiber
section extending over $R$, hence modulo $\eps^3$,
contradicting the claim. Thus $\mathcal H=0$. Flat base change to the fraction field of $R$,
followed by semicontinuity, gives
\[
 H^0(C,E_\rho(P_\eps)\otimes L)=0
\]
for general $\eps$. This proves
\eqref{eq:interp-canonical-vanishing} on a nonempty
open subset of $\MM_G^\tau\times C$.

We now apply the same assertion to
$(\rho^\vee,V^\vee)$. Complete reducibility gives
$(V^\vee)^G=0$, so Serre duality implies, on a
nonempty open subset,
\[
 H^1\bigl(C,E_\rho(P)((g-1)x)\bigr)
 \simeq
 H^0\bigl(C,E_{\rho^\vee}(P)
 \otimes K_C(-(g-1)x)\bigr)^\vee=0.
\]
 Riemann--Roch gives
\[
 \chi\bigl(E_\rho(P)((g-1)x)\bigr)
 =
 \dim V\,(g-1)+\dim V\,(1-g)=0.
\]
Consequently,
\[
 H^0\bigl(C,E_\rho(P)((g-1)x)\bigr)=0,
\]
proving the assertion in the semisimple case.

Now  suppose that $G$ is connected reductive,
with fixed central data, and that $\rho|_{Z(G)}=1$.
Then $\rho$ factors through the connected semisimple
group $G_{\mathrm{ad}}$, and $E_\rho(P)$ depends only
on the associated adjoint bundle.
Apply the semisimple assertion to $G_{\mathrm{ad}}$
and pull back the resulting vanishing open subset along
\[
 \MM_{G,Q}^{\tau}\times C
 \longrightarrow
 \MM_{G_{\mathrm{ad}}}^{\tau_{\mathrm{ad}}}\times C.
\]
Here $\tau_{\mathrm{ad}}$ is the type induced by $\tau$.
Lemma~\ref{lem:adjoint-reduction} makes this inverse
image nonempty and dense,  proving the
fixed-central-data assertion.
For $\MM_G^\tau\times C$, the vanishing locus is open by
semicontinuity. Any nonempty open substack contains a point
$(P,x)$ and has nonempty pullback to
$\MM_{G,Q}^\tau\times C$ for $Q=P\times^G G^{\mathrm{ab}}$.
The fixed-central-data assertion makes this pullback meet the
vanishing locus. Hence the vanishing locus is dense also without
fixing $Q$.

We return to the semisimple family constructed above to prove the
claim. Let $\varphi_0$ be the original gluing data of
$E_\rho(P_0)\otimes L$ from the $C^\circ$ side to
the disc side. A section of the deformed bundle is a pair satisfying
\begin{equation}\label{eq:interp-section-matching}
 \sigma_{D_p}(\eps)
 =
 \rho(\gamma(\eps,z))\,
 \varphi_0\bigl(\sigma_{C^\circ}(\eps)\bigr)
 \quad\text{on }D_p^\times.
\end{equation}

Suppose $0\ne v\in V^A$ and $sv$ lifts to a section
$\sigma(\eps)$ of $E_\rho(P_\eps)\otimes L$ modulo $\eps^3$.
Over $C^\circ$ and $D_p$, write this lift as

\begin{equation}\label{eq:interp-local-lifts}
 \begin{aligned}
 \sigma_{C^\circ}(\eps)
 &=(sv)|_{C^\circ}+\eps a_1+\eps^2a_2,\\
 \sigma_{D_p}(\eps)
 &=(sv)|_{D_p}+\eps b_1+\eps^2b_2,
 \end{aligned}
\end{equation}
where $a_i$ and $b_i$ are sections of $E_\rho(P_0)\otimes L$
on $C^\circ$ and $D_p$, respectively. Comparing
coefficients in \eqref{eq:interp-section-matching} gives
\begin{align}
 b_1-\varphi_0(a_1)
 &=suA_Xv,
 \label{eq:interp-first-order}\\
 b_2-\varphi_0(a_2)
 &=uA_X\varphi_0(a_1)+\frac12su^2A_X^2v.
 \label{eq:interp-second-order}
\end{align}

Since $su$ is regular at $p$, $(0,suA_Xv)$ is a first-order
solution. By Proposition~\ref{prop:principal-parts-matching},
the difference of any other solution and this one is a global
section of $E_\rho(P_0)\otimes L$. Equation
\eqref{eq:interp-special-cohomology} therefore gives $h\in V^A$ with
\begin{equation}\label{eq:interp-all-first-order-lifts}
 a_1=(sh)|_{C^\circ},\qquad
 b_1=suA_Xv+(sh)|_{D_p}.
\end{equation}

Substituting \eqref{eq:interp-all-first-order-lifts}
into \eqref{eq:interp-second-order} gives
\begin{equation}\label{eq:second-order}
b_2-\varphi_0(a_2)=suA_Xh+\frac12su^2A_X^2v.
\end{equation}

The first summand is regular at $p$, so its cohomology
class $[suA_Xh]\in H^1(C,E_\rho(P_0)\otimes L)$ vanishes by
Proposition~\ref{prop:principal-parts-duality}\textup{(i)}.
The projection in \eqref{eq:finite}, tensored with $L$, induces
the $H^1$-isomorphism in \eqref{eq:interp-special-cohomology},
by \eqref{eq:Lchoice}. In the chosen framing it is represented
by $\Pi$. Naturality of principal-parts classes
(Proposition~\ref{prop:principal-parts-duality}\textup{(iii)})
therefore gives, under that isomorphism,
\[
 \left[\frac12su^2A_X^2v\right]
 =\frac12\Pi A_X^2v\otimes[su^2]
 \in V^A\otimes H^1(C,L).
\]
This class is nonzero because $v\ne0$,
$\Pi A_X^2|_{V^A}$ is invertible, and $[su^2]\ne0$.
By Proposition~\ref{prop:principal-parts-matching},
\eqref{eq:second-order} has no solution. This contradiction
proves the claim and completes the proof.
\end{proof}

Return to the connected reductive group $G$ and representation
$(\rho,V)$ of Theorem~\ref{thm:vanishing}, with $r=\dim V$.
By that theorem and Serre duality, a general pair $(P,x)$ satisfies
\begin{equation}\label{eq:dualinterp}
 H^0\bigl(C,E_\rho(P)^\vee\otimes
 K_C(-(g-1)x)\bigr)=0.
\end{equation}

Fix a principal bundle $P$ for which \eqref{eq:dualinterp} holds
for general $x\in C$, and fix a subspace
$Z\subset H^0(C,E_\rho(P)^\vee\otimes K_C)$ of dimension $q$
before choosing $x$. Put
\begin{equation}\label{eq:ap}
 a_p(x)=\dim Z(-(p-1)x)-\dim Z(-px),
 \qquad p\geq1,
\end{equation}
where \[
 Z(-px):=Z\cap H^0(C,E_\rho(P)^\vee\otimes K_C(-px)).
\]

\begin{lemma}
\label{lem:jets}
For the fixed $P$ and $Z$ above and a general point $x\in C$, we have
\begin{equation}\label{eq:monotone}
 r\geq a_1(x)\geq a_2(x)\geq\cdots\geq a_{g-1}(x)\geq0,
 \qquad
 \sum_{p=1}^{g-1}a_p(x)=q.
\end{equation}
\end{lemma}

\begin{proof}
Put
\[
 F:=E_\rho(P)^\vee\otimes K_C,
 \qquad \rk F=r,
 \qquad \dim_{\C}Z=q.
\]
Recall that
\[
 Z(-mx)=Z\cap H^0(C,F(-mx))\quad(m\geq0),
 \qquad
 a_p(x)=\dim Z(-(p-1)x)-\dim Z(-px)\quad(p\geq1).
\]
If $q=0$, all the assertions are immediate, so assume $q>0$.
By the hypothesis on $P$ and semicontinuity, the set
\[
 U_0:=\bigl\{x\in C:H^0(C,F(-(g-1)x))=0\bigr\}
\]
is a nonempty Zariski open subset of $C$. We shall prove the
assertions on a possibly smaller nonempty open subset of $U_0$.

Choose a nonempty affine open subset $U\subset U_0$ on which $F$
has a frame $e_1,\ldots,e_r$. After shrinking $U$, choose a regular
function $z$ on $U$ such that $dz$ is nowhere zero.
Let $\partial=d/dz$ be the derivation characterized by
$dh=(\partial h)\,dz$ for local regular functions $h$.

The sections in $Z$ define the evaluation morphism
\[
 \operatorname{ev}_Z:
 Z\otimes_{\C}\OO_U\longrightarrow F|_U,
 \qquad s\otimes h\longmapsto h\,s|_U.
\]
Dualizing gives
\[
 \operatorname{ev}_Z^\vee:
 F^\vee|_U\longrightarrow
 \mathcal T:=Z^\vee\otimes_{\C}\OO_U.
\]
Derivatives of sections of $\mathcal T$ are taken coefficientwise
in its constant trivialization. If $e_1^\vee,\ldots,e_r^\vee$ is
the dual frame, define
\[
 f_i:=\operatorname{ev}_Z^\vee(e_i^\vee)
 \in H^0(U,\mathcal T),\qquad 1\leq i\leq r.
\]

Then for a  section $s\in Z$, we have
$s|_U=\sum_{i=1}^r f_i(s) e_i.$

Fix $x\in U$ and $p\geq1$. Let $\mathfrak m_x$ be the maximal ideal
of $\OO_{C,x}$, and let
\[
 px:=\operatorname{Spec}(\OO_{C,x}/\mathfrak m_x^p)
\]
be the length-$p$ infinitesimal neighbourhood of $x$.
The space of $(p-1)$-jets of $F$ at $x$ is
\[
 J_x^{p-1}(F):=H^0(px,F|_{px}),
\]
and jet evaluation is the restriction map
\[
 j_x^{p-1}:Z\longrightarrow J_x^{p-1}(F),
 \qquad s\longmapsto s|_{px}.
\]
Because $dz(x)\neq0$, the function $t=z-z(x)$ is a uniformizer at
$x$. In the chosen frame, Taylor expansion gives
\[
 j_x^{p-1}(s)
 =\sum_{i=1}^r\sum_{j=0}^{p-1}
   \frac{(\partial^j f_i)(x)(s)}{j!}\,t^j e_i
 \pmod{t^p}.
\]
Thus the $(p-1)$-jet records the Taylor coefficients of orders
$0,\ldots,p-1$. It is zero exactly when every coefficient $f_i(s)$
vanishes to order at least $p$ at $x$. Consequently,
\[
 \ker j_x^{p-1}=Z(-px),
\]
and
\[
 \im\bigl((j_x^{p-1})^\vee\bigr)
 =\operatorname{span}_{\C}
 \bigl\{(\partial^j f_i)(x):1\leq i\leq r,\ 0\leq j<p\bigr\}.
\]

Set $\mathcal S_0=0$ and, for $1\leq p\leq g-1$, define
\[
 \mathcal S_p
 :=\sum_{i=1}^r\sum_{j=0}^{p-1}
   \OO_U\,\partial^j f_i
 \subset\mathcal T.
\]
After shrinking $U$, we may assume that $\mathcal S_p$ and $\mathcal T/\mathcal S_p$
are locally free, and the fiber of $\mathcal S_p$ is the image
of the corresponding fiber map. Thus, for every $x\in U$,
\[
 \mathcal S_p|_x
 =\operatorname{span}_{\C}
 \bigl\{(\partial^j f_i)(x):1\leq i\leq r,\ 0\leq j<p\bigr\}
 \subset Z^\vee.
\]
Therefore
\[
 \mathcal S_p|_x
 =\im\bigl((j_x^{p-1})^\vee\bigr)
 =\Ann_{Z^\vee}\bigl(Z(-px)\bigr).
\]
Taking dimensions therefore gives
\[
 \rk\mathcal S_p
 =\dim(\mathcal S_p|_x)
 =q-\dim Z(-px).
\]
We conclude that
\[
 \begin{aligned}
 \rk(\mathcal S_p/\mathcal S_{p-1})
 &=\rk\mathcal S_p-\rk\mathcal S_{p-1}\\
 &=\bigl(q-\dim Z(-px)\bigr)
   -\bigl(q-\dim Z(-(p-1)x)\bigr)\\
 &=\dim Z(-(p-1)x)-\dim Z(-px)\\
 &=a_p(x).
 \end{aligned}
\]
In particular, the numbers $a_p(x)$ are constant on this open subset
$U$ and are nonnegative.

For $1\leq p<g-1$, differentiation induces a surjective
$\OO_U$-linear map
\[
 \psi:\mathcal S_p/\mathcal S_{p-1}
 \longrightarrow\mathcal S_{p+1}/\mathcal S_p,
 \qquad [v]\longmapsto[\partial v].
\]
It is well defined because
$\partial\mathcal S_{p-1}\subset\mathcal S_p$.
It is $\OO_U$-linear because the extra term in
$\partial(fv)=f\partial v+(\partial f)v$ lies in
$\mathcal S_p$. $\psi$ is surjective since classes of $\partial^p f_i$ generate the
target and are images of the classes of $\partial^{p-1}f_i$.
Thus $a_{p+1}(x)\leq a_p(x)$ on $U$.

Since $\mathcal S_1$ is generated by $r$ sections,
$a_1(x)=\rk\mathcal S_1\leq r$. Moreover, $U\subset U_0$, so
$Z(-(g-1)x)=0$. Summing \eqref{eq:ap} gives
$\sum_{p=1}^{g-1}a_p(x)=q$.
\end{proof}

\section{Hecke curves and splitting types}\label{sec:hecke}

As before, assume $g\geq3$. Let $(\rho,V)$ satisfy
$V^G=0$, $\rho|_{Z(G)}=1$ and $m_\rho\leq g-1$;
we still do not require it to be irreducible in this section.
We work on $\MM=\MM_{G,Q}^{\tau}$ and
$M=M_{G,Q}^{\tau}(C)$, and fix a simple factor $G_i$
of $G_{\mathrm{ad}}$. Notice that $m_i\leq m_\rho\leq g-1$.

\subsection{Complete highest root Hecke curves}

For $(P,\alpha)\in\MM_{G,Q}^{\tau}$ and $x\in C$, choose a uniformizer $z$ at $x$ and write $D_x=\operatorname{Spec}\C[[z]]$ and $D_x^\times=\operatorname{Spec}\C((z))$. Choose trivializations on $D_x$ for which the specified isomorphism $\alpha:P\times^G G^{\mathrm{ab}}\simeq Q$ becomes the identity:
\begin{equation}\label{framing}
\begin{tikzcd}
 P|_{D_x} \arrow[r,"\sim"] \arrow[d]
 & D_x\times G \arrow[d]\\
 Q|_{D_x} \arrow[r,"\sim"]
 & D_x\times G^{\mathrm{ab}}.
\end{tikzcd}
\end{equation}
After fixing the bottom trivialization, the choices of the top one form
a torsor under $G_{\mathrm{der}}(\C[[z]])$.

Put
\[
 LG_{\mathrm{der}}=G_{\mathrm{der}}(\C((z))),\qquad
 L^+G_{\mathrm{der}}=G_{\mathrm{der}}(\C[[z]]),\qquad
 \operatorname{Gr}_{G_{\mathrm{der}},x}
 =LG_{\mathrm{der}}/L^+G_{\mathrm{der}}.
\]
We use the same gluing direction as in
Section~\ref{sec:interpolation}. Let $\varphi$ denote the original
identification on $D_x^\times$, from the restriction of
$P|_{C\setminus\{x\}}$ to the trivialized bundle on $D_x$.
For $\gamma\in LG_{\mathrm{der}}$, replace $\varphi$ by
$\gamma^{-1}\varphi$. Beauville--Laszlo gluing~\cite{BL95}
gives a new principal $G$-bundle. Since $\gamma$ takes values
in $G_{\mathrm{der}}$, its $G^{\mathrm{ab}}$-quotient, including
its identification with $Q$, is unchanged. Replacing $\gamma$
by $\gamma h$, with $h\in L^+G_{\mathrm{der}}$, replaces the
new gluing by $h^{-1}\gamma^{-1}\varphi$ and hence only changes
the disc trivialization. We thus obtain
\begin{equation}\label{eq:relative-uniformization}
 \operatorname{unif}_{P,x}:
 \operatorname{Gr}_{G_{\mathrm{der}},x}^{0}
 \longrightarrow\MM_{G,Q}^{\tau}.
\end{equation}
This is the relative form of the gluing construction in
\cite[Definition~1.4 and Proposition~2.8]{KNR94}, with
$P|_{C\setminus\{x\}}$ kept fixed.
The superscript $0$ denotes the component of the identity coset; modifications in this component preserve $\tau$, and the identity coset maps to $P$ under this map.

Let $\theta_i$ be the highest long root of $G_i$ with respect to the chosen maximal torus and Borel subgroup. Choose nonzero root vectors
$e_{\theta_i}\in\Lie(G_i)_{\theta_i}$ and
$f_{\theta_i}\in\Lie(G_i)_{-\theta_i}$, normalized so that, for
$h_{\theta_i}=[e_{\theta_i},f_{\theta_i}]$,
\[
 [h_{\theta_i},e_{\theta_i}]=2e_{\theta_i},\qquad
 [h_{\theta_i},f_{\theta_i}]=-2f_{\theta_i}.
\]
Here we view these vectors in $\Lie(G_{\mathrm{der}})$ through its
identification with $\bigoplus_i\Lie(G_i)$. The corresponding
$\mathfrak{sl}_2$-homomorphism integrates, through the simply connected
cover of $G_{\mathrm{der}}$, to a homomorphism
\begin{equation}\label{eq:root-sl2-homomorphism}
 \iota_i:\SL_2\longrightarrow G_{\mathrm{der}}
\end{equation}
whose differential sends the standard upper, diagonal and lower $\mathfrak{sl}_2$ generators to $e_{\theta_i},h_{\theta_i},f_{\theta_i}$. The homomorphism $\iota_i$ also induces a homomorphism $L\SL_2\rightarrow LG_{\mathrm{der}}$, which we still denote by $\iota_i$.

Define the homomorphism
\begin{equation}\label{eq:root-loop-group}
 \begin{aligned}
 \kappa_z:\SL_2&\longrightarrow LG_{\mathrm{der}},\\
 \begin{pmatrix}a&b\\c&d\end{pmatrix}
 &\longmapsto
 \iota_i\!\begin{pmatrix}a&bz^{-1}\\cz&d\end{pmatrix}.
 \end{aligned}
\end{equation}
The image of the differential $d\kappa_z:\mathfrak{sl}_2\to\Lie(G_{\mathrm{der}})((z))$ is spanned by
\begin{equation}\label{eq:loops}
 z^{-1}e_{\theta_i},\qquad h_{\theta_i},\qquad zf_{\theta_i}.
\end{equation}
The inverse image $\kappa_z^{-1}(L^+G_{\mathrm{der}})$ is the lower triangular Borel $B^-\subset\SL_2$. In this paper, we fix an isomorphism $\SL_2/B^-\simeq\Pp^1$ so that the identity coset is $0\in \mathbb P^1$. The induced map to the affine Grassmannian, followed by \eqref{eq:relative-uniformization}, gives
\begin{equation}\label{eq:hecke-uniformization}
 \mathfrak f:\Pp^1=\SL_2/B^-
 \longrightarrow\operatorname{Gr}_{G_{\mathrm{der}},x}^{0}
 \xrightarrow{\operatorname{unif}_{P,x}}\MM_{G,Q}^{\tau}.
\end{equation}

Clearly $\mathfrak f(0)=P$. On the affine chart through $0$, the coset represented by $\left(\begin{smallmatrix}1&t\\0&1\end{smallmatrix}\right)$ gives
\begin{equation}\label{eq:hecke-affine-chart}
 \kappa_z\!\begin{pmatrix}1&t\\0&1\end{pmatrix}
 =\iota_i\!\begin{pmatrix}1&t/z\\0&1\end{pmatrix}
 =\exp(te_{\theta_i}/z).
\end{equation}
This is the root-curve construction used in \cite[proof of Proposition~2.1]{BGH21}.

\begin{definition}\label{def:highest-root-hecke}
We call \eqref{eq:hecke-uniformization} an \emph{$i$th highest-root Hecke family}. If its entire image lies in $\mathfrak U$ and its composite $f:\Pp^1\to U$ is nonconstant, we call $f$ an \emph{$i$th highest-root Hecke curve} and $x$ is the modification point of $f$.
\end{definition}

\begin{proposition}\label{prop:complete-hecke}
Fix $i$ and a closed subset $Z\subset U$ of codimension at least two. There exist an irreducible smooth variety $S$ and morphisms
\[
 \sigma:S\longrightarrow C,\qquad  F:S\times\Pp^1\longrightarrow U\setminus Z
\]
such that for each $s\in S$, $F_s$ is an $i$th highest-root Hecke curve with modification point $\sigma(s)$, and
\[
 S\longrightarrow U\times C,\qquad  s\longmapsto\bigl(F_s(0),\sigma(s)\bigr)
\]
is dominant. Moreover, for a general pair $(P,x)$, a single formal
framing satisfying \eqref{framing} can be chosen so that the
highest-root families associated with all simple factors give
nonconstant curves
\[
 f_i:\Pp^1\longrightarrow U\setminus Z,\qquad f_i(0)=[P],
\]
with modification point $x$.
\end{proposition}

\begin{proof}
Choose a nonempty affine open subset $C^\circ\subset C$ on which $Q$ is trivial, fix this trivialization, and choose a regular function
\[
 w:C^\circ\longrightarrow\mathbb A^1
\]
with nowhere vanishing differential. For $x\in C^\circ$, we use $z_x=w-w(x)$ as a uniformizer at $x$.

Choose a closed embedding $G_{\mathrm{der}}\hookrightarrow\GL_N$. For every simple factor, the entries of $\kappa_z(\eta)$ and $\kappa_z(\eta)^{-1}$ are Laurent polynomials in $z$ with regular coefficients on $\SL_2$. Choose a common bound $b$ for their pole
orders, valid for all factors, and take $a>2b$. Then
\[
 K_a=\ker\bigl(G_{\mathrm{der}}(\C[[z_x]])
       \longrightarrow G_{\mathrm{der}}(\C[z_x]/(z_x^a))\bigr)
\]
satisfies, for every simple factor,
\begin{equation}\label{eq:congruence-stabilizer}
 \kappa_{z_x}(\eta)^{-1}K_a\kappa_{z_x}(\eta)  \subset L^+G_{\mathrm{der}}  \qquad(\eta\in\SL_2).
\end{equation}
Indeed, the conjugate of a matrix congruent to the identity modulo $z_x^a$ is congruent to the identity modulo $z_x^{a-2b}$, and so is regular and invertible. Since the embedding is closed, the conjugate belongs to $G_{\mathrm{der}}(\C[[z_x]])$.

For a $\C$-algebra $R$ and a morphism $x_R:\operatorname{Spec}R\to C^\circ$, let $\Gamma_{x_R}\subset\operatorname{Spec}R\times C$ be its graph. On the formal neighbourhood of this graph, use
\[
 z_{x_R}=w(y)-w(x_R)
\]
as a relative parameter, where $y$ denotes the point of the second factor $C$. Since $w$ is \etale{}, this parameter identifies the subscheme defined by $\mathcal I_{\Gamma_{x_R}}^a$ with $\operatorname{Spec}(R[z_{x_R}]/(z_{x_R}^a))$.

Let $\mathscr S_a$ be the stack whose objects over $R$ are tuples $(P_R,\alpha_R,x_R,\psi_R)$, where $(P_R,\alpha_R)\in\MM(R)$, $x_R:\operatorname{Spec}R\to C^\circ$, and
\[
 \psi_R:  P_R|_{\operatorname{Spec}(R[z_{x_R}]/(z_{x_R}^a))}  \xrightarrow{\sim}  \operatorname{Spec}(R[z_{x_R}]/(z_{x_R}^a))\times G
\]
is a trivialization satisfying the restriction of \eqref{framing}. Here the restriction is taken on the thickening of $\Gamma_{x_R}$ just described, and the trivialization induced on $Q$ is the one fixed above. The forgetful morphism
\[
 \mathscr S_a\longrightarrow\MM\times C^\circ,  \qquad (P,\alpha,x,\psi)\longmapsto((P,\alpha),x),
\]
is a torsor under the smooth affine algebraic group representing
\begin{equation}\label{eq:finite-framing-functor}
 R\longmapsto G_{\mathrm{der}}(R[z]/(z^a)),
\end{equation}
base changed to $\MM\times C^\circ$. Here $z$ is an indeterminate,
identified with $z_{x_R}$ on the thickening of $\Gamma_{x_R}$.
In particular, the forgetful morphism and $\mathscr S_a$ are smooth,
using Proposition~\ref{prop:inputs}.

A trivialization $\psi_R$ lifts to a compatible formal framing by successive smooth lifting on the affine thickenings defined by powers of $z_{x_R}$. The matrix calculation above remains valid over $R[[z_{x_R}]]$: changing this lift by an element congruent to the identity modulo $z_{x_R}^a$ gives isomorphic modified bundles. These isomorphisms are the identity off $\Gamma_{x_R}$, satisfy the cocycle condition, and are compatible with base change. Consequently, Beauville--Laszlo gluing~\cite{BL95} gives, for every simple factor, a morphism locally of finite presentation
\begin{equation}\label{eq:framed-hecke-evaluation}
 \operatorname{ev}_i:
 \mathscr S_a\times\Pp^1\longrightarrow\MM.
\end{equation}
It sends $(P,\alpha,x,\psi,t)$ to the modification defined by \eqref{eq:hecke-uniformization}, and returns $(P,\alpha)$ at $t=0$. All factors use the same framing $\psi$.

We compute the differential of $\operatorname{ev}_i$.
In the following pointwise calculation, we write $z=z_x$.
Fix $(P,\alpha,x,\psi)\in\mathscr S_a$ and
$\eta\in\SL_2$, and put
\[
 P'=\operatorname{ev}_i(P,\alpha,x,\psi,\eta B^-),\qquad
 E_{\Ad}(P)=P\times^G\Lie(G_{\mathrm{der}}).
\]
Keeping $x$ and $\eta B^-$ fixed, the isomorphism classes of first order deformations of the framed bundle $(P,\alpha,\psi)$ and of the modified bundle $P'$ as an object of $\MM$ are parametrized, respectively, by
\[
 H^1\bigl(C,E_{\Ad}(P)(-ax)\bigr),  \qquad  H^1\bigl(C,E_{\Ad}(P')\bigr).
\]
Indeed, infinitesimal automorphisms preserving $\alpha$ are sections of $E_{\Ad}(P)$. To preserve $\psi$, they must also vanish on $ax$, so their sheaf is
\[
 \ker\bigl(E_{\Ad}(P)\longrightarrow E_{\Ad}(P)|_{ax}\bigr)
 =E_{\Ad}(P)(-ax).
\]
The usual \v{C}ech description of first order deformations therefore gives the first cohomology group above; applying the same description without the framing condition gives the second.

Choose a formal framing lifting $\psi$. Differentiating \eqref{eq:congruence-stabilizer} at the identity gives
\[
 \Ad\bigl(\kappa_z(\eta)^{-1}\bigr)
 \bigl(z^a\Lie(G_{\mathrm{der}})[[z]]\bigr)
 \subset\Lie(G_{\mathrm{der}})[[z]].
\]
Thus the identification of $P$ and $P'$ away from $x$ extends to an inclusion
\[
 E_{\Ad}(P)(-ax)\lhook\joinrel\longrightarrow E_{\Ad}(P').
\]
In the source and target disc framings, this inclusion is given by $\Ad(\kappa_z(\eta)^{-1})$. It is an isomorphism away from $x$, so its quotient is supported at $x$.

We verify that the induced map on $H^1$ is the differential of
$\operatorname{ev}_i$ with $x$ and $\eta B^-$ fixed. Let $\varphi$ be
the original gluing from $C\setminus\{x\}$ to the disc, and
represent a first order deformation of $(P,\alpha,\psi)$ by
\[
 u\in\Lie(G_{\mathrm{der}})((z))
\]
through the gluing $(1-\varepsilon u)\varphi$, with
$\varepsilon^2=0$. The minus sign agrees with the
section-matching convention $b-\varphi(a)$ in
\eqref{eq:principal-parts-matching-map}.
Applying the fixed modification gives
\[
 \kappa_z(\eta)^{-1}(1-\varepsilon u)\varphi
 =
 \bigl(1-\varepsilon\Ad(\kappa_z(\eta)^{-1})u\bigr)
 \kappa_z(\eta)^{-1}\varphi.
\]
Hence the induced first order deformation of $P'$ is
represented by $\Ad(\kappa_z(\eta)^{-1})u$. This is precisely
the map on \v{C}ech cocycles induced by the displayed inclusion.

Since the quotient of the displayed inclusion is supported at $x$,
its first cohomology vanishes. The cohomology exact sequence
therefore gives a surjection
\[
 H^1\bigl(C,E_{\Ad}(P)(-ax)\bigr)
 \twoheadrightarrow H^1\bigl(C,E_{\Ad}(P')\bigr).
\]
Thus the differential of $\operatorname{ev}_i$ is surjective at
every geometric point, already on first order deformations
keeping both $x$ and the point of $\Pp^1$ fixed.
Since both $\mathscr S_a\times\Pp^1$ and $\MM$ are smooth over
$\C$ and $\operatorname{ev}_i$ is locally of finite presentation,
the tangent space criterion for smoothness shows that
$\operatorname{ev}_i$ is smooth.

Choose a nonempty irreducible affine smooth chart $S\to\mathscr S_a$ of finite type, and let $\sigma:S\to C^\circ$ record the modification point. The morphism
\[
 S\longrightarrow\MM\times C^\circ,\qquad
 s\longmapsto((P_s,\alpha_s),\sigma(s)),
\]
is smooth. Its image is nonempty and open, hence dense by irreducibility of $\MM\times C^\circ$. Theorem~\ref{thm:interp} for the adjoint action on $\Lie(G_{\mathrm{der}})$ gives
\[
 H^0\bigl(C,E_{\Ad}(P)(x)\bigr)=0
\]
on a nonempty open subset of $\MM\times C^\circ$, since $x\leq(g-1)x$. Replace $S$ by the inverse image of this open subset; it is nonempty. After this restriction, for every $s\in S$, the source bundle $P_s$ and the modification point $\sigma(s)$ satisfy
\[
 H^0\bigl(C,E_{\Ad}(P_s)(\sigma(s))\bigr)=0.
\]

Let $\mathscr D\subset\MM$ be the union of $\MM\setminus\mathfrak U$ and the closure in $\MM$ of the inverse image of $Z$ in $\mathfrak U$. By Proposition~\ref{prop:inputs}, its codimension is at least two. Use the same notation $\operatorname{ev}_i$ for the pulled back maps on $S\times\Pp^1$. Their smoothness gives
\[
 \dim\operatorname{ev}_i^{-1}(\mathscr D)
 \leq\dim(S\times\Pp^1)-2=\dim S-1.
\]
The projection $S\times\Pp^1\to S$ is proper. The image of each $\operatorname{ev}_i^{-1}(\mathscr D)$ is therefore closed and has dimension at most $\dim S-1$. Remove the union of these images for all simple factors and retain the notation $S$ for the complement. There are only finitely many factors, so this complement is nonempty. For every $s\in S$, every simple factor $i$, and every $t\in\Pp^1$, the modified bundle $\operatorname{ev}_i(s,t)$ belongs to $\mathfrak U$, and its isomorphism class in $U$ does not belong to $Z$. Thus each evaluation morphism
\[
 \operatorname{ev}_i:S\times\Pp^1\longrightarrow\MM
\]
factors through $\mathfrak U$, and its composite with the coarse moduli morphism $\mathfrak U\to U$ factors through $U\setminus Z$.

For each pair $(P,x)$ parametrized by the remaining $S$, the vanishing imposed above makes the connecting map
\[
 H^0\bigl(C,E_{\Ad}(P)(x)/E_{\Ad}(P)\bigr)
 \lhook\joinrel\longrightarrow H^1(C,E_{\Ad}(P))
 \simeq T_{[P]}U
\]
injective. By \eqref{eq:hecke-affine-chart}, the differential at
$0$ of the curve for the $i$th factor is the image of the
nonzero principal part $z^{-1}e_{\theta_i}$, expressed in the
source framing. Hence every resulting curve is nonconstant.

For the index $i$ fixed in the statement, compose
$\operatorname{ev}_i|_{S\times\Pp^1}$ with the coarse moduli map
to obtain
\[
 F:S\times\Pp^1\longrightarrow U\setminus Z.
\]
The variety $S$ is smooth because $\mathscr S_a$ is smooth, and
all the restrictions made above are to nonempty open subsets.
The source morphism factors through $\mathfrak U\times C^\circ$;
composing with the smooth morphism
$\mathfrak U\times C^\circ\to U\times C$ from
Proposition~\ref{prop:inputs} gives
\[
 S\longrightarrow U\times C,\qquad
 s\longmapsto\bigl(F_s(0),\sigma(s)\bigr).
\]
This map is smooth with nonempty open, hence dense, image.
This proves the first assertion.

Finally, for a general pair $(P,x)$, choose a point of $S$ with
source isomorphic to $(P,\alpha,x)$. The single finite order
framing carried by this point was used for all simple factors.
Transport it to $P$ and lift it to a formal framing satisfying
\eqref{framing}. By \eqref{eq:congruence-stabilizer} and the
base change compatible gluing isomorphisms above, this lift
recovers all the families at the chosen point of $S$. They are
simultaneously nonconstant, complete, and disjoint from $Z$.
This proves the simultaneous assertion.
\end{proof}

We now normalize the classes $\Theta_i$ from
Proposition~\ref{prop:inputs}. For an $i$th highest-root Hecke
curve, the projection to the $j$th adjoint moduli space is
constant when $j\ne i$. For $j=i$, the pullback of the ample
generator has the fixed positive degree of its restriction
to the highest-root $\Pp^1$ in the affine Grassmannian; see
\cite[Lemma~1.4, Proposition~1.5 and \S2.4]{BLS98}.
This degree is independent of the choice of $P$ with $f_i(0)=[P]$, the modification point $x$
and the framing in \eqref{framing}. Rescaling each $\Theta_i$
by a positive rational number, we obtain
\begin{equation}\label{eq:dual-bases}
 \deg f_i^*\Theta_j=\delta_{ij}.
\end{equation}
We use this normalization from now on.

\subsection{The associated family and the bundle \texorpdfstring{$E^+$}{E+}}

Let $f:\Pp^1\to U$ be a complete $i$th highest-root Hecke curve constructed from $P$, $x$, $z$ and a framing satisfying \eqref{framing}, and let $\mathfrak f:\Pp^1\to\mathfrak U$ be its lift from \eqref{eq:hecke-uniformization}. By Proposition~\ref{prop:associated-descent}, the vector bundle
associated with $\mathfrak f$ is the pullback of $\EE_\rho$ along
$f\times\operatorname{id}_C$. Write
\begin{equation}\label{eq:hecke-associated-family}
 \EE_f=(f\times \operatorname{id}_C)^*\EE_\rho,  \qquad E_t=\EE_f|_{\{t\}\times C},\qquad E_0=E_\rho(P).
\end{equation}

Write $q_C:\Pp^1\times C\to C$ and
$\pi:\Pp^1\times C\to\Pp^1$ for the projections. We compute
the splitting type of $R^1\pi_*\EE_f$. Once we have shown that
$f(\Pp^1)\subset U_\rho$, cohomology and base change, together
with $\mathcal V_\rho|_{U_\rho}\simeq\mathcal V_\rho^\circ$, give
\[
 f^*\mathcal V_\rho\simeq f^*\mathcal V_\rho^\circ
 \simeq R^1\pi_*\EE_f.
\]

The framing in \eqref{framing} identifies the regular sections of
$E_\rho(P)$ on $D_x$ with
\[
 \Lambda_0=V[[z]]\subset V((z)).
\]
Pull back $\EE_f$ first to $\Pp^1\times D_x$, and then along
\[
 \SL_2\times D_x\longrightarrow \Pp^1\times D_x,  \qquad (\eta,z)\longmapsto(\eta B^-,z).
\]
The framing in \eqref{framing} trivializes this pullback. In the fixed meromorphic identification away from $x$, its sections are embedded by $v\mapsto\rho(\kappa_z(\eta))v$. Thus, at $t=\eta B^-$, the lattice is
\[
 \Lambda_t=\rho(\kappa_z(\eta))\Lambda_0.
\]
For $b\in B^-$, the loop $\kappa_z(b)$ is regular and preserves
$\Lambda_0$. Replacing $\eta$ by $\eta b$ changes the disc coordinates by
$\rho(\kappa_z(b))^{-1}$. Consequently,
\begin{equation}\label{eq:hecke-disc-descent}
 \EE_f|_{\Pp^1\times D_x}
 \simeq\SL_2\times^{B^-}\Lambda_0,
 \qquad
 (\eta,v)\cdot b=(\eta b,\rho(\kappa_z(b))^{-1}v).
\end{equation}
The right-hand side denotes the sheaf obtained by descent of the
$\C[[z]]$-module $\Lambda_0$; its meromorphic embedding is
$[\eta,v]\mapsto\rho(\kappa_z(\eta))v$.

Put
\begin{equation}\label{eq:factor-data}
 \nu_i=d\rho(e_{\theta_i}).
\end{equation}
Fix an $\SL_2$-equivariant decomposition
\begin{equation}\label{eq:root-representation-decomposition}
 V\simeq\bigoplus_{\ell}\Sym^{a_\ell}\C^2,
 \qquad 0\leq a_\ell\leq m_i,
\end{equation}
for $\rho\circ\iota_i$; see \cite{Hum72}. We use this fixed isomorphism
to regard each summand as a subspace of $V$, and write
$v_{\ell,0},\ldots,v_{\ell,a_\ell}$ for its standard monomial basis.
Extending $\rho$ to $\C((z))$ and composing with
\eqref{eq:root-loop-group} gives the action
$\eta\cdot v=\rho(\kappa_z(\eta))v$ on $V((z))$.
The vectors $v_{\ell,k}$ form a $\C((z))$-basis of $V((z))$, and the
differential of this action satisfies
\begin{equation}\label{eq:root-action-original-basis}
 \begin{aligned}
 (z^{-1}\nu_i)(z^p v_{\ell,k})
   &=kz^{p-1}v_{\ell,k-1},\\
 d\rho(h_{\theta_i})(z^p v_{\ell,k})
   &=(a_\ell-2k)z^p v_{\ell,k},\\
 (z\,d\rho(f_{\theta_i}))(z^p v_{\ell,k})
   &=(a_\ell-k)z^{p+1}v_{\ell,k+1},
 \end{aligned}
 \qquad p\in\mathbb Z,\quad 0\leq k\leq a_\ell,
\end{equation}
with zero at the relevant endpoints. These operators are
$\C((z))$-linear, so the displayed formulas determine their action
on arbitrary Laurent series. In this action, the matrices
$\left(\begin{smallmatrix}1&t\\0&1\end{smallmatrix}\right)$ and
$\left(\begin{smallmatrix}1&0\\c&1\end{smallmatrix}\right)$ act respectively by
\[
 \exp(tz^{-1}\nu_i),\qquad \exp(cz\,d\rho(f_{\theta_i})),
\]
and $\operatorname{diag}(s,s^{-1})$ sends $z^p v_{\ell,k}$ to
$s^{a_\ell-2k}z^p v_{\ell,k}$. In particular, the lower triangular
group $B^-$ preserves $\Lambda_0$.

A lattice in $V((z))$ is a free $\C[[z]]$-submodule of rank $r$ which spans $V((z))$ over $\C((z))$. We seek a lattice $\Lambda^+$ satisfying
\begin{equation}\label{eq:containing-lattice-conditions}
 \Lambda_0\subset\Lambda^+\subset z^{-m_i}\Lambda_0,
 \qquad
 \rho(\kappa_z(\eta))\Lambda^+=\Lambda^+
 \quad(\eta\in\SL_2).
\end{equation}

\begin{lemma}\label{lem:common-containing-bundle}
Suppose $m_i\leq g-1$ and $(P,x)$ satisfies
\begin{equation}\label{vanishing}
 H^0\bigl(C,E_\rho(P)((g-1)x)\bigr)  =H^1\bigl(C,E_\rho(P)((g-1)x)\bigr)=0
\end{equation}
for $\rho$.
There is a vector bundle $E^+$ on $C$ such that
\begin{equation}\label{eq:common-containing-bundle}
 E_\rho(P)\subset E^+\subset E_\rho(P)(m_i x)
 \subset E_\rho(P)((g-1)x),
\end{equation}
and we have an exact sequence on $\Pp^1\times C$
\begin{equation}\label{eq:Heckeexact}
 0\longrightarrow\EE_f\longrightarrow q_C^*E^+
 \longrightarrow\mathcal Q\longrightarrow0,
\end{equation}
where the quotient $\mathcal Q=q_C^*E^+/\EE_f$ is supported on
$\Pp^1\times(m_i x)$ and is flat over $\Pp^1$.
Moreover, $f(\Pp^1)\subset U_\rho$.
\end{lemma}

\begin{proof}
As lattices in $V((z))$, we set
\begin{equation}\label{eq:superlattice}
 \Lambda^+=V[[z]]+
       \sum_{p=1}^{m_i}z^{-p}\im(\nu_i^p)[[z]],
 \qquad \Lambda_0=V[[z]].
\end{equation}
This is the $\C[[z]]$-module generated by repeatedly applying
$z^{-1}\nu_i$ to $\Lambda_0$. Since $\nu_i$ preserves every summand
in \eqref{eq:root-representation-decomposition}, both lattices are
the direct sums of their intersections with the corresponding
Laurent series spaces. Explicitly, for every $\ell$,
\begin{equation}\label{eq:sl2-containing-lattices}
 \begin{aligned}
 \Lambda^+\cap\bigl(\Sym^{a_\ell}\C^2\bigr)((z))
   &=\bigoplus_{k=0}^{a_\ell}
       \C[[z]]z^{-(a_\ell-k)}v_{\ell,k},\\
 \Lambda_0\cap\bigl(\Sym^{a_\ell}\C^2\bigr)((z))
   &=\bigoplus_{k=0}^{a_\ell}\C[[z]]v_{\ell,k}.
 \end{aligned}
\end{equation}
In particular, $\Lambda^+$ is a lattice and satisfies the first pair
of inclusions in \eqref{eq:containing-lattice-conditions}.
For $p\geq1$ one has
\[
 d\rho(f_{\theta_i})\im\nu_i^p\subset\im\nu_i^{p-1}:
\]
indeed, under \eqref{eq:root-representation-decomposition},
\[
 \im\nu_i^p\cap\Sym^{a_\ell}\C^2
 =\langle v_{\ell,0},\ldots,v_{\ell,a_\ell-p}\rangle
 \qquad(p\leq a_\ell),
\]
and this intersection is zero when $p>a_\ell$.
Therefore
\[
 \begin{aligned}
 (z^{-1}\nu_i)\bigl(z^{-p}\im\nu_i^p[[z]]\bigr)
   &\subset z^{-(p+1)}\im\nu_i^{p+1}[[z]],\\
 (z\,d\rho(f_{\theta_i}))
       \bigl(z^{-p}\im\nu_i^p[[z]]\bigr)
   &\subset z^{-(p-1)}\im\nu_i^{p-1}[[z]].
 \end{aligned}
\]
The second operator also preserves $V[[z]]$, and the diagonal subgroup preserves every displayed image space. The two unipotent actions are finite exponentials of these operators. Hence the whole $\SL_2$ acting through $\rho\circ\kappa_z$ preserves $\Lambda^+$, proving \eqref{eq:containing-lattice-conditions}. It follows that
\[
 \Lambda_t=\rho(\kappa_z(\eta))\Lambda_0
 \subset\rho(\kappa_z(\eta))\Lambda^+=\Lambda^+.
\]

Glue $\Lambda^+$ to $E_\rho(P)|_{C\setminus\{x\}}$ by the original
meromorphic identification. Beauville--Laszlo gluing~\cite{BL95}
gives a vector bundle $E^+$ and the inclusions
\eqref{eq:common-containing-bundle}.

The inclusion $\Lambda_0\subset\Lambda^+$ is $B^-$-equivariant.
Since $\rho\circ\kappa_z$ preserves $\Lambda^+$, the map
\[
 \SL_2\times^{B^-}\Lambda^+\longrightarrow\Pp^1\times\Lambda^+,
 \qquad [\eta,w]\longmapsto(\eta B^-,\rho(\kappa_z(\eta))w)
\]
is an isomorphism. Composing it with the inclusion obtained from \eqref{eq:hecke-disc-descent}, and with the fixed identification away from $x$, gives the global injection $\EE_f\hookrightarrow q_C^*E^+$ and hence \eqref{eq:Heckeexact}.

Since $z^{m_i}\Lambda^+\subset\Lambda_0$ and $\Lambda^+$ is
invariant, $z^{m_i}\Lambda^+\subset\Lambda_t$ for every $t$.
Thus $\mathcal Q$ has the claimed support. Locally on $\Pp^1$,
choose a regular representative $\eta$ of $t=\eta B^-$. Multiplication
by $\rho(\kappa_z(\eta))^{-1}$ identifies $\Lambda^+/\Lambda_t$ with the constant finite-dimensional module
$\Lambda^+/\Lambda_0$. This proves flatness over $\Pp^1$. Finally, the vanishing \eqref{vanishing} at $(P,x)$ and the
inclusions \eqref{eq:common-containing-bundle} give
\begin{equation}\label{eq:hecke-family-vanishing}
 H^0(C,E^+)=H^0(C,E_t)=0\qquad(t\in\Pp^1),
\end{equation}
so the image of $f$ lies in $U_\rho$.
\end{proof}

\subsection{Splitting types and the positive degree filtration}

We compute $\pi_*\mathcal Q$ from the $B^-$-action on
$\Lambda^+/\Lambda_0$. If $m_i=0$, then $\Lambda^+=\Lambda_0$ and $\mathcal Q=0$. Otherwise, since $\mathcal Q$ is supported on $\Pp^1\times(m_i x)$, it is determined by its restriction to $\Pp^1\times D_x$. The preceding descent identifies this restriction with $\SL_2\times^{B^-}(\Lambda^+/\Lambda_0)$, where $\Lambda^+/\Lambda_0$ is viewed as a $\C[z]/(z^{m_i})$-module with the $B^-$-action induced by $\rho\circ\kappa_z$. Pushing forward along $\pi$ forgets this module structure and gives
\begin{equation}\label{eq:quotient-associated-bundle}
 \pi_*\mathcal Q\simeq\SL_2\times^{B^-}(\Lambda^+/\Lambda_0),
\end{equation}
where the quotient on the right is now viewed as a finite-dimensional $B^-$-representation over $\C$. This formula also holds when $m_i=0$, with both sides zero.

For $d\in\mathbb Z$, let $\C_d$ be the one-dimensional $B^-$-representation
on which $\left(\begin{smallmatrix}s&0\\c&s^{-1}\end{smallmatrix}\right)$
acts by $s^d$. Our convention gives
\begin{equation}\label{eq:character-associated-line}
 \SL_2\times^{B^-}\C_d\simeq\OO_{\Pp^1}(d).
\end{equation}
The $B^-$-action on $\Sym^a\C^2$ is the restriction of its $\SL_2$-action. The associated vector bundle is therefore trivial. Tensoring with \eqref{eq:character-associated-line} gives
\begin{equation}\label{eq:twisted-sl2-associated-bundle}
 \SL_2\times^{B^-}       \bigl(\C_d\otimes\Sym^a\C^2\bigr)  \simeq\OO_{\Pp^1}(d)^{\oplus(a+1)}.
\end{equation}

For the fixed simple factor $G_i$, recall that $\nu_i=d\rho(e_{\theta_i})$ as in~\eqref{eq:factor-data}. Put

\begin{equation}\label{eq:hecke-numerical-data}
 d_{i,j}=\rk\nu_i^j,\qquad
 \beta_{i,j}=\sum_{k\geq j}d_{i,k},\qquad
 e_i=\sum_{j\geq1}j\,d_{i,j}.
\end{equation}

\begin{lemma}\label{lem:sl2-local-quotient}
With $d_{i,j}$ as in \eqref{eq:hecke-numerical-data}, we have
\begin{equation}\label{eq:B}
 \pi_*\mathcal Q\simeq
 \bigoplus_{j=1}^{m_i}\OO_{\Pp^1}(j)^{\oplus d_{i,j}}.
\end{equation}
\end{lemma}

\begin{proof}
Use the fixed decomposition \eqref{eq:root-representation-decomposition}
and the intersections in \eqref{eq:sl2-containing-lattices}. The full
quotient $\Lambda^+/\Lambda_0$ has the $\C$-basis
\[
 \overline{z^{-p}v_{\ell,k}},\qquad
 a_\ell\geq1,\quad 0\leq k\leq a_\ell-1,\quad
 1\leq p\leq a_\ell-k.
\]
We construct an isomorphism of $B^-$-representations
\begin{equation}\label{eq:source-quotient-layer}
 \Lambda^+/\Lambda_0\xrightarrow{\sim}
 \bigoplus_{\ell}\bigoplus_{j=1}^{a_\ell}
 \left(\C_j\otimes\Sym^{a_\ell-j}\C^2\right).
\end{equation}
Writing $\mathbf e_0,\mathbf e_1$ for the standard basis of $\C^2$, we define
the image of $\overline{z^{-p}v_{\ell,k}}$ in the
$(\ell,j)$-summand of \eqref{eq:source-quotient-layer}, with
$j=a_\ell-p-k+1$, by
\begin{equation}\label{eq:sl2-quotient-module-isomorphism}
 \overline{z^{-p}v_{\ell,k}}
 \longmapsto
 \frac{(a_\ell-k)!}{(p-1)!}
 \left(1\otimes \mathbf e_0^{p-1}\mathbf e_1^k\right).
\end{equation}
The indexing conditions imply $1\leq j\leq a_\ell$ and
$p-1+k=a_\ell-j$. Conversely, every monomial in the $(\ell,j)$
component arises uniquely from such a basis vector. All displayed
coefficients are nonzero, so this map is a vector-space isomorphism.

We check equivariance. The source vector has diagonal weight
$a_\ell-2k$, and its image has weight
\[
 j+(p-1-k)=a_\ell-2k.
\]
The lower root generator acts on the source by
\[
 (z\,d\rho(f_{\theta_i}))\overline{z^{-p}v_{\ell,k}}
  =(a_\ell-k)\overline{z^{-(p-1)}v_{\ell,k+1}}.
\]
For $p=1$ this is zero modulo $\Lambda_0$, as is the lower action on
$\mathbf e_1^k$. For $p\geq2$ it belongs to the same $(\ell,j)$ component,
and the required equality of coefficients is
\[
 (a_\ell-k)\frac{(a_\ell-k-1)!}{(p-2)!}
 =(p-1)\frac{(a_\ell-k)!}{(p-1)!}.
\]
Thus \eqref{eq:sl2-quotient-module-isomorphism} intertwines the diagonal
and lower unipotent subgroups, and hence all of $B^-$.

Combining \eqref{eq:source-quotient-layer} with
\eqref{eq:quotient-associated-bundle} and
\eqref{eq:twisted-sl2-associated-bundle} gives
\begin{equation}\label{eq:sl2-quotient-splitting}
 \pi_*\mathcal Q\simeq
 \bigoplus_{\ell}\bigoplus_{j=1}^{a_\ell}
 \OO_{\Pp^1}(j)^{\oplus(a_\ell-j+1)}.
\end{equation}
For each $j$, the multiplicity is $\sum_{\ell}\max\{a_\ell-j+1,0\} =\rk\nu_i^j=d_{i,j}$, which proves \eqref{eq:B}.
\end{proof}

Now assume the hypotheses of
Lemma~\ref{lem:common-containing-bundle}. The vanishing
\eqref{eq:hecke-family-vanishing} and $\deg E_t=0$ imply that
$R^1\pi_*\EE_f$ is locally free of rank $(g-1)r$.
Since $f(\Pp^1)\subset U_\rho$, the identification from
Section~\ref{sec:moduli} and cohomology and base change give
\begin{equation}\label{eq:hecke-base-change}
 f^*\mathcal V_\rho\simeq f^*\mathcal V_\rho^\circ
 \simeq R^1\pi_*\EE_f.
\end{equation}
The support of $\mathcal Q$ is finite over $\Pp^1$, so $R^1\pi_*\mathcal Q=0$. Pushing forward \eqref{eq:Heckeexact} on $\Pp^1\times C$ gives an exact sequence on $\Pp^1$:
\begin{equation}\label{eq:hecke-cohomology-sequence}
 0\longrightarrow \pi_*\mathcal{Q}\longrightarrow f^*\mathcal V_\rho^\circ  \longrightarrow H^1(C,E^+)\otimes\OO_{\Pp^1}\longrightarrow0.
\end{equation}
At $0\in \Pp^1$, the first map is the connecting map
\begin{equation}\label{eq:source-connecting}
 \delta_0:\Lambda^+/\Lambda_0
 \simeq H^0(C,E^+/E_\rho(P))
 \lhook\joinrel\longrightarrow H^1(C,E_\rho(P)).
\end{equation}

For $j\geq1$, define the positive degree subbundles and their fibers at $0$ by
\begin{equation}\label{eq:positive-filtration-definition}
 \begin{aligned}
 B_{i,\geq j}
 &=\im\!\left(
 H^0\bigl(\Pp^1,(f^*\mathcal V_\rho^\circ)(-j)\bigr)
 \otimes\OO_{\Pp^1}(j)
 \longrightarrow f^*\mathcal V_\rho^\circ\right),\\
 T_{i,j}&=B_{i,\geq j}|_0\subset H^1(C,E_\rho(P)).
 \end{aligned}
\end{equation}
For any splitting on $\Pp^1$, this image is precisely the sum of the line bundle summands of degree at least $j$. It is therefore a subbundle independent of the choice of splitting.

\begin{proposition}\label{prop:splitting}
Let $f$ be a complete $i$th highest-root Hecke curve satisfying the hypotheses of Lemma~\ref{lem:common-containing-bundle}. Then
\begin{equation}\label{eq:splitting}
 f^*\mathcal V_\rho^\circ\simeq
 \left(\bigoplus_{j=1}^{m_i}
       \OO_{\Pp^1}(j)^{\oplus d_{i,j}}\right)
 \oplus\OO_{\Pp^1}^{\oplus((g-1)r-\sum_jd_{i,j})}.
\end{equation}
In particular, $\deg f^*\mathcal V_\rho^\circ=e_i$, and the injection
in \eqref{eq:hecke-cohomology-sequence} identifies $\pi_*\mathcal Q$
with $B_{i,\geq1}$. For $1\leq j\leq m_i$,
\begin{equation}\label{eq:Tj}
 T_{i,j}=\delta_0\!\left(
    \bigoplus_{p=1}^{m_i-j+1}z^{-p}\im\nu_i^{p+j-1}\right),
 \qquad \dim T_{i,j}=\beta_{i,j}.
\end{equation}
\end{proposition}

\begin{proof}
Every summand in \eqref{eq:B} has positive degree, so
$H^1(\Pp^1,\pi_*\mathcal Q)=0$. Since the right-hand term in
\eqref{eq:hecke-cohomology-sequence} is trivial, its extension class
lies in
\[
 H^1(C,E^+)^\vee\otimes H^1(\Pp^1,\pi_*\mathcal Q)=0.
\]
The sequence therefore splits. Its middle term has rank $(g-1)r$,
giving \eqref{eq:splitting} and the degree formula. Every positive degree summand maps trivially to the quotient in
\eqref{eq:hecke-cohomology-sequence}, so
$B_{i,\geq1}\subset\pi_*\mathcal Q$.
Both are subbundles of $f^*\mathcal V_\rho^\circ$ of the same rank,
hence $B_{i,\geq1}=\pi_*\mathcal Q$.

We identify $\Lambda^+/\Lambda_0$ and each summand's principal-parts
space with their natural images in $V((z))/\Lambda_0$.
By \eqref{eq:sl2-quotient-module-isomorphism}, the class
$\overline{z^{-p}v_{\ell,k}}$ contributes to degree
$a_\ell-p-k+1$ on $\Pp^1$. Since the injection
\eqref{eq:hecke-cohomology-sequence} identifies the positive subbundle
with $\pi_*\mathcal Q$, for every $\ell$ we obtain
\begin{equation}\label{eq:sl2-positive-fiber}
 \begin{aligned}
 \delta_0^{-1}(T_{i,j})\cap \frac{\bigl(\Sym^{a_\ell}\C^2\bigr)((z))}{\bigl(\Sym^{a_\ell}\C^2\bigr)[[z]]} = \bigoplus_{p=1}^{a_\ell-j+1} z^{-p}\langle v_{\ell,0},\ldots,v_{\ell,a_\ell-p-j+1}\rangle,
 \end{aligned}
\end{equation}
where all terms are principal-part classes and the sum is empty if
$a_\ell<j$. For each term,
\[
 \langle v_{\ell,0},\ldots,v_{\ell,a_\ell-p-j+1}\rangle
 =\im\nu_i^{p+j-1}\cap\Sym^{a_\ell}\C^2.
\]
Taking the direct sum over all $\ell$ and applying $\delta_0$ gives
\eqref{eq:Tj}. Its dimension is
\[
 \sum_{p=1}^{m_i-j+1}d_{i,p+j-1}
 =\sum_{k=j}^{m_i}d_{i,k}=\beta_{i,j}.
\]
\end{proof}

\begin{remark}\label{rem:splitting-beyond}
The construction of $\Lambda^+$, $E^+$ and
\eqref{eq:Heckeexact} in Lemma~\ref{lem:common-containing-bundle},
and the calculation of $\pi_*\mathcal Q$ in
Lemma~\ref{lem:sl2-local-quotient}, are independent of
\eqref{eq:interp} and do not require the bound $m_i\leq g-1$.
This bound is used only to deduce $H^0(C,E^+)=0$
from \eqref{eq:interp}. For a complete root family in $U$, any
other proof of this vanishing therefore gives the injectivity of
\eqref{eq:source-connecting}, the vanishing of every $H^0(C,E_t)$,
and the formulas \eqref{eq:splitting} and \eqref{eq:Tj}.
Here the injective map has domain $H^0(C,E^+/E_\rho(P))$,
rather than the full principal-parts space
$H^0(C,E_\rho(P)(m_i x)/E_\rho(P))$.
\end{remark}

\subsection{Codimension and the first Chern class}
\label{subsec:cohomology-extension}

Throughout this subsection, assume $g\geq3$, $V^G=0$,
$\rho|_{Z(G)}=1$ and $m_\rho\leq g-1$.
For either $M=M_{G,Q}^{\tau}(C)$ or $M=M_G^\tau(C)$, retain
$\mathcal V_\rho$, $U_\rho$ and $j_\rho$ from
Section~\ref{sec:moduli}. In particular,
\[
 \mathcal V_\rho|_{U_\rho}\simeq\mathcal V_\rho^\circ
 =(R^1p_*\EE_\rho)|_{U_\rho}.
\]
Cohomology and base change, together with Serre duality, gives,
for $[P]\in U_\rho$,
\begin{equation}\label{eq:Vr-fibers}
 (\mathcal V_\rho^\circ)_{[P]}=H^1(C,E_\rho(P)),\qquad
 (\mathcal V_\rho^\circ)^\vee_{[P]}
 =H^0(C,E_\rho(P)^\vee\otimes K_C).
\end{equation}
For $M=M_{G,Q}^{\tau}(C)$, the preceding construction gives
complete curves in $U_\rho$ without a codimension estimate for
$M\setminus U_\rho$. We use these curves to prove that
$\codim_M(M\setminus U_\rho)\geq2$, and hence that
$\mathcal V_\rho\simeq(j_\rho)_*\mathcal V_\rho^\circ$.
Tensoring with $Z(G)^0$-bundles then gives the assertion for
$M_G^\tau(C)$.

\begin{proposition}\label{prop:bigopen}\label{prop:cohomology-sheaf}
Take $M=M_{G,Q}^{\tau}(C)$. Then
\[
 \codim_M(M\setminus U_\rho)\geq2.
\]
There is a canonical identification
\begin{equation}\label{eq:Wrho-definition}
 \mathcal V_\rho\simeq(j_\rho)_*\mathcal V_\rho^\circ.
\end{equation}
The sheaf $\mathcal V_\rho$ has rank $(g-1)r$, and
\begin{equation}\label{eq:Wrho-c1-reductive}
 c_1(\mathcal V_\rho)=\sum_{i=1}^s e_i\Theta_i.
\end{equation}
For the adjoint representation on $\Lie(G_{\mathrm{der}})$,
the corresponding construction gives $T_M$.
\end{proposition}

\begin{proof}
Theorem~\ref{thm:interp} and Lemma~\ref{lem:adjoint-reduction} show
that $U_\rho$ is nonempty. By semicontinuity, $U\setminus U_\rho$ is
closed in $U$. Suppose it has a component of codimension one, and let
$A$ be that component's closure in $M$. Then $A$ is a prime divisor
and $A\cap U\subset U\setminus U_\rho$.
For each $i$, Proposition~\ref{prop:complete-hecke} and
Theorem~\ref{thm:interp} give a complete Hecke curve $f_i$
whose pair $(P,x)$ satisfies \eqref{eq:interp}. Since
$m_i\leq g-1$, Lemma~\ref{lem:common-containing-bundle} gives
$f_i(\Pp^1)\subset U_\rho$. This also holds when $m_i=0$,
in which case the associated vector bundle family is constant.

By Proposition~\ref{prop:inputs}, $\Cl(M)_{\mathbb Q}$ is spanned by
rational Cartier classes, so $A$ is $\mathbb Q$-Cartier. Each $f_i$
is disjoint from $A$; the canonical section of a Cartier multiple of
$A$ therefore trivializes its pullback. Consequently $A\cdot f_i=0$
for every $i$. By \eqref{eq:dual-bases}, $[A]=0$ in
$\Cl(M)_{\mathbb Q}$, contrary to $A\cdot H^{\dim M-1}>0$ for an
ample Cartier divisor $H$.
It follows that $U\setminus U_\rho$ has codimension at least two;
Proposition~\ref{prop:inputs} gives the same assertion in $M$.

Since $\mathcal V_\rho$ is coherent and reflexive and restricts
to $\mathcal V_\rho^\circ$ on $U_\rho$, the codimension estimate
makes the canonical map
\[
 \mathcal V_\rho\longrightarrow(j_\rho)_*\mathcal V_\rho^\circ
\]
an isomorphism by \cite[Tag~0EBJ]{Stacks}.
This proves \eqref{eq:Wrho-definition}; its rank is $(g-1)r$.

The base change identification \eqref{eq:hecke-base-change}
and Proposition~\ref{prop:splitting} give
$\deg f_i^*\mathcal V_\rho=e_i$; trivial factors have degree zero.
The dual divisor basis then gives \eqref{eq:Wrho-c1-reductive}.

For $V=\Lie(G_{\mathrm{der}})$ with
$\rho(h)=\Ad(h)|_{\Lie(G_{\mathrm{der}})}$, we have
$E_\rho(P)=E_{\Ad}(P)$. Infinitesimal automorphisms preserving $Q$
vanish on $U$, and the deformation space is $H^1(C,E_{\Ad}(P))$.
Thus $R^1p_*\EE_\rho\simeq T_U$.
The sheaves $\mathcal V_\rho$ and
$T_M=\mathcal Hom(\Omega_M^1,\OO_M)$ are reflexive and agree on $U$,
so they agree on $M$.
\end{proof}

\begin{proposition}\label{prop:varying-central-data}
Take $M=M_G^\tau(C)$. Then
\[
 \codim_M(M\setminus U_\rho)\geq2.
\]
The identification \eqref{eq:Wrho-definition} holds for this
moduli space as well, and $\rk\mathcal V_\rho=(g-1)r$.
\end{proposition}

\begin{proof}
Fix a $G^{\mathrm{ab}}$-bundle $Q$ of type $\tau^{\mathrm{ab}}$.
The isogeny $Z(G)^0\to G^{\mathrm{ab}}$ induces a finite
\etale{} surjection
\[
 \begin{aligned}
 M_{Z(G)^0}^{0}(C)
 &\longrightarrow M_{G^{\mathrm{ab}}}^{\tau^{\mathrm{ab}}}(C),\\
 [R]&\longmapsto
 [Q\otimes(R\times^{Z(G)^0}G^{\mathrm{ab}})].
 \end{aligned}
\]
Indeed, Smith normal form on the character lattices identifies
the induced map on degree zero bundles with a product of
multiplication maps on $\Pic^0(C)$.

For a $Z(G)^0$-bundle $R$, define $P\otimes R$ by extension of
structure group along $G\times Z(G)^0\to G$, $(h,z)\mapsto hz$.
For $R$ of type zero, this operation preserves semistability,
S-equivalence and $\tau$. Normalized Poincar\'e families realize
it as an algebraic action of $M_{Z(G)^0}^{0}(C)$ on $M$, giving
\[
 \begin{aligned}
 \beta:M_{G,Q}^{\tau}(C)\times M_{Z(G)^0}^{0}(C)
 &\longrightarrow M,\\
 ([P],[R])&\longmapsto[P\otimes R].
 \end{aligned}
\]
On $M\times_{M_{G^{\mathrm{ab}}}^{\tau^{\mathrm{ab}}}(C)}
M_{Z(G)^0}^{0}(C)$, tensoring by $R^{-1}$ defines the inverse
\[
 ([P],[R])\longmapsto([P\otimes R^{-1}],[R])
\]
to $(\beta,\operatorname{pr}_2)$, since
$M_{G,Q}^{\tau}(C)=\operatorname{ab}^{-1}([Q])$.
Thus $\beta$ is the base change of the displayed finite
\etale{} surjection.

Central tensoring leaves the adjoint bundle unchanged, so
$\beta^{-1}(U)$ is the product of the corresponding open subset
of $M_{G,Q}^{\tau}(C)$ with $M_{Z(G)^0}^{0}(C)$.
Since $M_{Z(G)^0}^{0}(C)$ is smooth,
Proposition~\ref{prop:inputs} and finite \etale{} descent show
that $M$ is normal, $U$ is smooth and
$\codim_M(M\setminus U)\geq2$.
These conclusions and the construction of $\beta$ are independent
of $\rho$.

Since $\rho|_{Z(G)}=1$, there is a canonical identification
\[
 E_\rho(P\otimes R)\simeq E_\rho(P),
\]
compatible with families. Hence $\beta^{-1}(U_\rho)$ is likewise
the product of the corresponding open subset of
$M_{G,Q}^{\tau}(C)$ with $M_{Z(G)^0}^{0}(C)$.
Proposition~\ref{prop:cohomology-sheaf} and finite \etale{}
surjectivity give $\codim_M(M\setminus U_\rho)\geq2$.
Since $\mathcal V_\rho|_{U_\rho}\simeq\mathcal V_\rho^\circ$,
reflexive extension gives \eqref{eq:Wrho-definition}
\cite[Tag~0EBJ]{Stacks}. Its rank is
$\rk\mathcal V_\rho^\circ=(g-1)r$.

On $\beta^{-1}(U_\rho)$, the identification of associated bundles
and cohomology and base change give
$\beta^*\mathcal V_\rho\simeq
\operatorname{pr}_1^*\mathcal V_\rho$.
Both pullbacks are reflexive because $\beta$ and
$\operatorname{pr}_1$ are flat \cite[Tag~0EBF]{Stacks}.
As the complement has codimension at least two, this
isomorphism extends uniquely to
\[
 \beta^*\mathcal V_\rho\simeq\operatorname{pr}_1^*\mathcal V_\rho
\]
on $M_{G,Q}^{\tau}(C)\times M_{Z(G)^0}^{0}(C)$, where the
sheaf on the right is the one on $M_{G,Q}^{\tau}(C)$.
\end{proof}

\section{The quotient-degree estimate and stability}
\label{sec:slope}\label{sec:intersection}

Assume the hypotheses of Theorem~\ref{thm:main}. We first take
$M=M_{G,Q}^{\tau}(C)$. Recall that
\[
 r=\dim V,\qquad \nu_i=d\rho(e_{\theta_i}),\qquad
 d_{i,j}=\rk\nu_i^j,\qquad e_i=\sum_{j=1}^{m_i}j\,d_{i,j}.
\]
We set $d_{i,j}=0$ for $j>m_i$. The classes $\Theta_i$ satisfy
\eqref{eq:dual-bases}, and Proposition~\ref{prop:cohomology-sheaf}
gives
\[
 \rk\mathcal V_\rho=(g-1)r,\qquad
 c_1(\mathcal V_\rho)=\sum_i e_i\Theta_i.
\]

\subsection{Proof of the main theorem}

\begin{proof}[Proof of Theorem~\ref{thm:main} for $M=M_{G,Q}^{\tau}(C)$]
Let $0\ne\FF\subsetneq\mathcal V_\rho$ be a saturated subsheaf of
smaller rank, and write
\begin{equation}\label{eq:bi}
 q=\rk(\mathcal V_\rho/\FF),\qquad
 c_1(\mathcal V_\rho/\FF)=\sum_i b_i\Theta_i.
\end{equation}
Then $0<q<(g-1)r$. By \eqref{eq:dual-bases},
\[
 b_i=\deg f_i^*(\mathcal V_\rho/\FF)
\]
for every complete $i$th Hecke curve $f_i:\Pp^1\to U_\rho$ avoiding the non-locally-free loci of $\FF$ and $\mathcal V_\rho/\FF$. Thus $b_i$ is independent of the choice of such a curve.

\smallskip\noindent\textbf{Claim.}
\begin{equation}\label{eq:factor-quotient-estimate}
 \frac{b_i}{q}\geq\frac{e_i}{(g-1)r}\quad\text{for every }i,
 \qquad
 \frac{b_i}{q}>\frac{e_i}{(g-1)r}\quad\text{for at least one }i.
\end{equation}

Then for every ample Cartier divisor $H$, the claim, \eqref{eq:Wrho-c1-reductive}, and \eqref{eq:theta-positive} give
\[
 \begin{aligned}
 \mu_H(\mathcal V_\rho/\FF)-\mu_H(\mathcal V_\rho) =\sum_i
 \left(\frac{b_i}{q}-\frac{e_i}{(g-1)r}\right)
 \Theta_i\cdot H^{\dim M-1}>0.
 \end{aligned}
\]
Equivalently, $\mu_H(\FF)<\mu_H(\mathcal V_\rho)$. This proves stability for every ample Cartier polarization.

Now we prove the claim. Choose a general $P\in U_\rho$ where $\FF$ and $\mathcal V_\rho/\FF$ are locally free, and use Serre duality to set
\[
 Z=\Ann(\FF|_{[P]})\subset H^0(C,E_\rho(P)^\vee \otimes K_C),\qquad  Z^\vee=(\mathcal V_\rho/\FF)|_{[P]},\qquad \dim Z=q.
\]
After fixing $P$ and $Z$, choose $x\in C$ generally so that \eqref{eq:interp}, Lemma~\ref{lem:jets}, and Proposition~\ref{prop:complete-hecke} apply. We also fix a uniformizer $z$ around $x$ with formal disc $D_x$ as before.

By the final assertion of Proposition~\ref{prop:complete-hecke},
choose a single formal framing of $P$ on $D_x$ satisfying
\eqref{framing} such that, for every $i$, the corresponding
highest-root curve $f_i:\Pp^1\to U_\rho$ is complete and avoids
the non-locally-free loci of $\FF$ and $\mathcal V_\rho/\FF$.
To vary these curves, let $h\in G_{\mathrm{ad}}$, choose a lift
to $G_{\mathrm{der}}$, and conjugate $\iota_i$ in
\eqref{eq:root-sl2-homomorphism} and $\kappa_z$ in
\eqref{eq:root-loop-group} by this lift.

The construction is independent of this lift, since any two lifts differ by a central element. It preserves the fixed bundle $Q$ and the source $P$. In the associated representation, on the affine chart through $0$, this replaces
\[
 \exp(t\nu_i/z)
 \quad\text{by}\quad
 \rho(h)\exp(t\nu_i/z)\rho(h)^{-1}
 =\exp\bigl(t\rho(h)\nu_i\rho(h)^{-1}/z\bigr).
\]
Throughout the variation in $h$, we use the same formal trivialization of $E_\rho(P)$ near $x$.

Such conjugation gives, for each $i$, an algebraic family of
principal bundles parametrized by $G_{\mathrm{der}}\times\Pp^1$.
The inverse image of the closed substack formed by
$\MM\setminus\mathfrak U$ and the closure in $\MM$ of the inverse
image in $\mathfrak U$ of the prescribed non-locally-free loci
is closed in $G_{\mathrm{der}}\times\Pp^1$. Its projection to
$G_{\mathrm{der}}$ is closed because $\Pp^1$ is proper. Taking
the complement of the union of these projected closed subsets
over all $i$ gives a Zariski open subset of $G_{\mathrm{der}}$.
It contains the identity, since the original curves $f_i$
already lie in $U$ and avoid the prescribed loci.

We also verify that the resulting maps to $U$ are nonconstant.
When choosing the general pair $(P,x)$ above, we impose the
vanishing
\[
 H^0\bigl(C,E_{\Ad}(P)(x)\bigr)=0,
 \qquad E_{\Ad}(P)=P\times^G\Lie(G_{\mathrm{der}}).
\]
This follows from Theorem~\ref{thm:interp} for the adjoint
action on $\Lie(G_{\mathrm{der}})$, since $x\leq(g-1)x$. As in the proof of
Proposition~\ref{prop:complete-hecke}, the exact sequence
\[
 0\longrightarrow E_{\Ad}(P)
 \longrightarrow E_{\Ad}(P)(x)
 \longrightarrow E_{\Ad}(P)(x)/E_{\Ad}(P)
 \longrightarrow0
\]
therefore gives an injective connecting map
\[
 H^0\bigl(C,E_{\Ad}(P)(x)/E_{\Ad}(P)\bigr)
 \lhook\joinrel\longrightarrow H^1(C,E_{\Ad}(P))
 \simeq T_{[P]}U.
\]
By \eqref{eq:hecke-affine-chart}, the conjugated modifying loop
near $0$ is
\[
 h\exp(te_{\theta_i}/z)h^{-1}
 =\exp\bigl(tz^{-1}\Ad(h)e_{\theta_i}\bigr).
\]
Its Kodaira--Spencer class at $0$ is the image of the principal
part $z^{-1}\Ad(h)e_{\theta_i}$ under this connecting map.
Since $\Ad(h)e_{\theta_i}\ne0$, this principal part is nonzero,
and injectivity makes its image nonzero as well. Under the
displayed tangent-space identification, this is the differential
of the induced map to $U$ at $0$. Thus every resulting map is
nonconstant.

Conjugation is unchanged when $h$ is multiplied by an element
of $Z(G_{\mathrm{der}})$. Hence the open subset above is
invariant under this finite central subgroup and descends,
under $G_{\mathrm{der}}\to G_{\mathrm{ad}}$, to a nonempty
Zariski open subset containing the identity. For every $h$
in this descended open subset, all the resulting curves are
complete and nonconstant and avoid the prescribed loci.
Their images lie in $U_\rho$ by
Lemma~\ref{lem:common-containing-bundle}.

For $h$ in this open subset, write
$f_{i,h}:\Pp^1\to U_\rho$ for the resulting curves, with
$f_{i,h}(0)=[P]$. We will choose $h$ generally below. The curves
$f_{i,h}$, the subspaces $T_{i,j}$, and the ranks $t_{i,j}$ of
their maps to $Z^\vee$ below may depend on $h$, whereas $P,x,z,Z$ and the
numbers $b_i$ remain fixed. Thus we have the exact sequence
of vector bundles on $\Pp^1$
\[
 0\longrightarrow f_{i,h}^*\FF
 \longrightarrow f_{i,h}^*\mathcal V_\rho
 \longrightarrow f_{i,h}^*(\mathcal V_\rho/\FF)
 \longrightarrow0.
\]

Thus $Z^\vee=f_{i, h}^*(\mathcal V_\rho/\FF)|_0$. Proposition~\ref{prop:splitting} shows that in the exact sequence, the middle term, and hence its quotient, are globally generated. Write
\[
 f_{i, h}^*(\mathcal V_\rho/\FF) \simeq\bigoplus_{\alpha=1}^q\OO_{\Pp^1}(\ell_\alpha), \qquad \ell_\alpha\geq0.
\]
For $j\geq1$, retain $T_{i,j}=B_{i,\geq j}|_0\subset H^1(C,E_\rho(P))$ from \eqref{eq:positive-filtration-definition}, and put
\[
 t_{i,j}=\rk\bigl(T_{i,j}\longrightarrow Z^\vee\bigr).
\]
A summand of degree at least $j$ in the source cannot map nontrivially to a summand of degree less than $j$ in the quotient. Thus $t_{i,j}\leq\#\{\alpha:\ell_\alpha\geq j\}$, and
\begin{equation}\label{eq:factor-degree-bound}
 b_i=\sum_{\alpha=1}^q\ell_\alpha
     =\sum_{j\geq1}\#\{\alpha:\ell_\alpha\geq j\}
     \geq\sum_{j=1}^{m_i}t_{i,j}.
\end{equation}

We express $t_{i,j}$ as the rank of a residue pairing. Assume $m_i>0$, and let
$\Lambda^+$ be the lattice of \eqref{eq:superlattice}, with $\Lambda_0=V[[z]]$. For the curve determined
by $h$, the local quotient is $\rho(h)\Lambda^+/\Lambda_0$.
Its connecting map is the restriction of the fixed map
\[
 \partial^{E_\rho(P)}_{x,g-1}:
 z^{-(g-1)}\Lambda_0/\Lambda_0\longrightarrow H^1(C,E_\rho(P))
\]
from \eqref{eq:principal-parts-boundary}. Multiplication by
$\rho(h)$ identifies $\Lambda^+/\Lambda_0$ with this translated
quotient. Define
\begin{equation}\label{eq:residue-pairing}
 \begin{aligned}
 (\Lambda^+/\Lambda_0)\times Z&\longrightarrow\C,\\
 (\overline w,s)&\longmapsto
 \langle\overline w,s\rangle_h
 :=\Res_x\bigl(s(\rho(h)w)\bigr).
 \end{aligned}
\end{equation}
This is well defined since $\rho(h)\Lambda_0=\Lambda_0$. By the
formula \eqref{eq:principal-parts-residue-pairing}, it equals
$\langle\partial^{E_\rho(P)}_{x,g-1}(\overline{\rho(h)w}),s\rangle$,
where the latter is
the Serre pairing.

For the curve $f_{i, h}$ determined by $h$, \eqref{eq:Tj} gives
\[
 T_{i,j}=\partial^{E_\rho(P)}_{x,g-1}\!\left(
 \bigoplus_{p=1}^{m_i-j+1}
 z^{-p}\rho(h)\im\nu_i^{p+j-1}\right).
\]
So $t_{i,j} =\rk\langle T_{i,j},Z\rangle$ and we have
\begin{equation}\label{eq:positive-pairing-rank}
 t_{i,j}=\rk\left\langle    \bigoplus_{p=1}^{m_i-j+1}z^{-p}\im\nu_i^{p+j-1},Z   \right\rangle_h.
\end{equation}

On $D_x$, use the fixed framing and the trivialization of $K_C$
by $dz$ to write a section $s\in Z$ as
\begin{equation}\label{expression of elements in Z}
     s=(s_0+zs_1+z^2s_2+\cdots)\,dz,\qquad s_k\in V^\vee.
\end{equation}
Then
\begin{equation}\label{pairing}
     \langle z^{-k}v,s\rangle_h=s_{k-1}(\rho(h)v).
\end{equation}

By the choice of $x$ and Lemma~\ref{lem:jets}, we have a filtration
\[
0=Z(-(g-1)x)\subset Z(-(g-2)x)\subset \cdots\subset Z(-x)\subset Z
\]
with
\[
 Z(-px)=Z\cap H^0(C,E_\rho(P)^\vee\otimes K_C(-px)).
\]
For this fixed $x$, write $a_p=a_p(x)$ for the increments in
\eqref{eq:ap}.
Moreover,
\begin{equation}\label{eq:ap-reductive}
 r\geq a_1\geq\cdots\geq a_{g-1}\geq0,\qquad  \sum_{p=1}^{g-1}a_p=q.
\end{equation}
So~\eqref{pairing} gives
\begin{equation}\label{eq:residue-filtration-vanishing}
 \left\langle z^{-k}\im\nu_i^{k+j-1},Z(-(p-1)x) \right\rangle_h=0\qquad(k<p).
\end{equation}
We therefore have induced pairings on the corresponding graded pieces:
\begin{equation}\label{eq:residue-graded-pairing}
 \begin{aligned}
 z^{-p}\im\nu_i^{p+j-1}\times
 \frac{Z(-(p-1)x)}{Z(-px)}&\longrightarrow\C,\\
 (z^{-p}v,[s])&\longmapsto s_{p-1}(\rho(h)v),
 \end{aligned}
 \qquad 1\leq p\leq m_i-j+1.
\end{equation}

In bases adapted to these filtrations, the matrix of \eqref{eq:positive-pairing-rank}, with rows ordered by vanishing order and columns by pole order, is block upper triangular. Its diagonal blocks are \eqref{eq:residue-graded-pairing}. Hence
\begin{equation}\label{eq:graded-pairing-rank-bound}
 t_{i,j}\geq
 \sum_{p=1}^{m_i-j+1}
 \rk\left\langle
 z^{-p}\im\nu_i^{p+j-1},
 \frac{Z(-(p-1)x)}{Z(-px)}
 \right\rangle_h,
\end{equation}
where the pairings on the right are the induced pairings
\eqref{eq:residue-graded-pairing}.

Taking the coefficient of $z^{p-1}dz$ in
\eqref{expression of elements in Z} gives
\[
 \frac{Z(-(p-1)x)}{Z(-px)}\xrightarrow{\sim}U_p\subset V^\vee,
 \qquad [s]\longmapsto s_{p-1},\qquad \dim U_p=a_p.
\]
The $p$th summand on the right-hand side of
\eqref{eq:graded-pairing-rank-bound} is the rank of
\[
 U_p\longrightarrow
 \bigl(\rho(h)\im\nu_i^{p+j-1}\bigr)^\vee.
\]
Its dual has kernel
$U_p^\perp\cap\rho(h)\im\nu_i^{p+j-1}$, where
\[
 U_p^\perp=\{v\in V:u(v)=0\text{ for every }u\in U_p\},
 \qquad \dim U_p^\perp=r-a_p.
\]
We apply Lemma~\ref{lem:one-factor-strict} to $U_p^\perp$ and $\im\nu_i^{p+j-1}$ in the irreducible $G_{\mathrm{ad}}$-module $V$, so for general $h$ we have
\[
 \begin{aligned}
 &\rk\left(
 U_p\longrightarrow
 (\rho(h)\im\nu_i^{p+j-1})^\vee\right)\\
 &\quad=d_{i,p+j-1}
   -\dim\bigl(U_p^\perp\cap\rho(h)\im\nu_i^{p+j-1}\bigr)\\
 &\quad\geq d_{i,p+j-1}
   -\frac{(r-a_p)d_{i,p+j-1}}r
   =\frac{a_p d_{i,p+j-1}}r.
 \end{aligned}
\]
There are finitely many such rank conditions for all $i,j,p$. Their dense open loci meet the nonempty open locus of $h$ giving the required complete curves. Thus, for general $h$ in that intersection, \eqref{eq:graded-pairing-rank-bound} gives
\begin{equation}\label{eq:factor-rank-bound}
 \begin{aligned}
 t_{i,j}
 &\geq\sum_{p=1}^{m_i-j+1}
 \rk\left(U_p\longrightarrow
 (\rho(h)\im\nu_i^{p+j-1})^\vee\right)\\
 &\geq\frac1r\sum_{p=1}^{m_i-j+1}a_p d_{i,p+j-1}.
 \end{aligned}
\end{equation}

Combining \eqref{eq:factor-degree-bound} and \eqref{eq:factor-rank-bound}, and collecting terms with $k=p+j-1$, yields
\begin{equation}\label{eq:weighted-i}
 b_i\geq\frac1r
 \sum_{j=1}^{m_i}\sum_{p=1}^{m_i-j+1}a_p d_{i,p+j-1}
 =\frac1r\sum_{k=1}^{m_i}d_{i,k}\sum_{p=1}^k a_p.
\end{equation}
By \eqref{eq:ap-reductive}, the average of the first $k$ terms
is at least the average of the whole sequence. Hence
\[
 \sum_{p=1}^k a_p\geq\frac{kq}{g-1}
 \qquad(1\leq k\leq g-1).
\]
Since $m_i\leq g-1$, we obtain
\begin{equation}\label{eq:nonstrict-i}
 b_i\geq\frac{q}{(g-1)r}\sum_{k=1}^{m_i}k\,d_{i,k}
 =\frac{qe_i}{(g-1)r}.
\end{equation}
For a trivial factor, we have $m_i=e_i=0$, so the same inequality is simply
$b_i\geq0$, already given by \eqref{eq:factor-degree-bound}.

Suppose equality holds in \eqref{eq:nonstrict-i} for every $i$.
There is a nontrivial factor, and for it $d_{i,1}>0$. Equality
in \eqref{eq:weighted-i} and the subsequent comparison forces
$a_1=q/(g-1)$. Together with \eqref{eq:ap-reductive}, this gives
\[
 0<a_1=\cdots=a_{g-1}=\frac{q}{g-1}<r.
\]
Thus $U_1^\perp$ is nonzero and proper. For each nontrivial factor,
\[
 \im\nu_i=
 \im\bigl(d\rho_i(e_{\theta_i})\bigr)
 \otimes\bigotimes_{k\ne i}V_k,
\]
and the first tensor factor is nonzero and proper. Apply
Lemma~\ref{lem:simultaneous-strictness} to $U_1^\perp$ and these
subspaces. Choosing $h$ also in the dense open locus supplied
by that lemma, while retaining all previous open conditions,
gives for at least one $i$
\[
 \rk\left(U_1\longrightarrow(\rho(h)\im\nu_i)^\vee\right)
 >\frac{a_1d_{i,1}}r.
\]
This is precisely the $(p,j)=(1,1)$ graded pairing. Its strict
inequality makes \eqref{eq:factor-rank-bound} strict for $j=1$
and that factor. Consequently
\[
 b_i>\frac1r\sum_{k=1}^{m_i}d_{i,k}\sum_{p=1}^k a_p
 =\frac{qe_i}{(g-1)r},
\]
contrary to the assumed equality. The numbers $b_i$ are unchanged
by this choice of $h$, so the claim follows.
\end{proof}

\begin{proof}[Proof of Theorem~\ref{thm:main} for $M=M_G^\tau(C)$]
Take $M=M_G^\tau(C)$ and use $Q$ and $\beta$ from the proof
of Proposition~\ref{prop:varying-central-data}. The final
identification in that proof gives
\[
 \beta^*\mathcal V_\rho\simeq\operatorname{pr}_1^*\mathcal V_\rho,
\]
where the sheaf on the right is the one on $M_{G,Q}^{\tau}(C)$.

If $G$ is semisimple, the assertion is the case already proved.
Assume $Z(G)^0\ne1$ and let $H$ be an ample Cartier divisor on $M$.
The Picard calculation in Proposition~\ref{prop:inputs} gives
$\Pic^0(M_{G,Q}^{\tau}(C))=0$. Hence
\[
 \beta^*H\equiv\operatorname{pr}_1^*H_1+
                 \operatorname{pr}_2^*H_2
\]
for ample Cartier divisors $H_1$ on $M_{G,Q}^{\tau}(C)$ and $H_2$
on $M_{Z(G)^0}^{0}(C)$. Indeed, the restrictions of $\OO(\beta^*H)$ to the fibers
over $M_{Z(G)^0}^{0}(C)$ all have the same class. After tensoring by the inverse
of its restriction to one such fiber, pulled back by
$\operatorname{pr}_1$, the resulting line bundle is trivial on
every fiber. Cohomology and base change identifies it with the
pullback of a line bundle on $M_{Z(G)^0}^{0}(C)$. Restriction to the two factors
shows that both resulting divisor classes are ample.

Let
\[
 0\ne\FF\subsetneq\operatorname{pr}_1^*\mathcal V_\rho
\]
be a saturated subsheaf of smaller rank. Generic flatness applied
to this inclusion and its quotient allows us to restrict it to
general fibers of either projection, preserving the inclusion
and the ranks.
We also choose these fibers so that they meet the
non-locally-free loci of the sheaves only in codimension
at least two. Their determinant classes therefore agree
with the restrictions of the corresponding first Chern classes.
For general $[R]\in M_{Z(G)^0}^{0}(C)$, the already proved
stability on $M_{G,Q}^{\tau}(C)$ gives
\[
 \mu_{H_1}\bigl(
 \FF|_{M_{G,Q}^{\tau}(C)\times\{[R]\}}\bigr)
 <\mu_{H_1}(\mathcal V_\rho).
\]
For general $[P]\in M_{G,Q}^{\tau}(C)$, the restriction of
$\operatorname{pr}_1^*\mathcal V_\rho$ to
$\{[P]\}\times M_{Z(G)^0}^{0}(C)$ is trivial. Therefore
\[
 \mu_{H_2}\bigl( \FF|_{\{[P]\}\times M_{Z(G)^0}^{0}(C)}\bigr)\leq0:
\]
a nonzero component of the exterior-power inclusion gives a
nonzero section of the dual determinant line bundle.
In computing the slope difference with respect to $\beta^*H$,
the only terms in the expansion of $(\beta^*H)^{\dim M-1}$ that
can contribute have exponents
\[
 \bigl(\dim M_{G,Q}^{\tau}(C)-1,\dim M_{Z(G)^0}^{0}(C)\bigr)
 \quad\text{or}\quad
 \bigl(\dim M_{G,Q}^{\tau}(C),\dim M_{Z(G)^0}^{0}(C)-1\bigr)
\]
on the two factors. Their contributions are positive multiples
of the two slope differences just computed: the first is strictly
negative and the second is nonpositive. Thus
\[
 \mu_{\beta^*H}(\FF)
 <\mu_{\beta^*H}(\operatorname{pr}_1^*\mathcal V_\rho).
\]
This proves stability on the product. Finally, a subsheaf on $M$
violating stability would pull back under $\beta$ to one violating
this inequality, since finite \etale{} pullback multiplies all
$H$-slopes by $\deg\beta$. Hence $\mathcal V_\rho$ is $H$-stable.
\end{proof}

\subsection{Intersection estimates in irreducible representations}

\begin{lemma}\label{lem:one-factor-strict}
Let $G$ be connected and let $(\rho,V)$ be an irreducible representation. For subspaces
$A,B\subset V$ and general $h\in G$,
\begin{equation}\label{eq:generic-intersection}
 \dim(A\cap\rho(h)B)\leq\frac{\dim A\,\dim B}{\dim V}.
\end{equation}
If $G$ is almost simple and $A,B$ are nonzero and proper, the inequality
is strict. In this almost simple case, the only endomorphisms preserving
every $\rho(h)B$ are scalars.
\end{lemma}

\begin{proof}

Fix $B\subset V$. For each subspace $S\subset V$, we define $\alpha(S)$ to be the value of $\dim(S\cap \rho(h)B)$ for generic $h$. Set $\alpha(0)=0$. Then for fixed $S,S'\subset V$ and any $h\in G$, we have
\[
 \begin{aligned}
 &(S\cap\rho(h)B)+(S'\cap\rho(h)B)
       \subset (S+S')\cap\rho(h)B,\\
 &(S\cap\rho(h)B)\cap(S'\cap\rho(h)B)
       =(S\cap S')\cap\rho(h)B.
 \end{aligned}
\]
Choose $h$ in the intersection of the four open subsets associated with $S,S',S+S'$ and $S\cap S'$. We obtain
\begin{equation}\label{eq:intersection-supermodular}
 \alpha(S)+\alpha(S')\leq \alpha(S+S')+\alpha(S\cap S').
\end{equation}
The function $\alpha$ is also invariant under translates of $S$.

Now choose a nonzero subspace $S$ maximizing $\alpha(S)/\dim S$, and among all such subspaces choose one of maximal dimension.

Fix any $h_0\in G$ and put $S'=\rho(h_0)S$. So $S'$ has the same dimension and the same generic intersection dimension as $S$. The maximality of the ratio and \eqref{eq:intersection-supermodular} imply
\[
 \begin{aligned}
 \alpha(S)+\alpha(S')
 &\leq\alpha(S+S')+\alpha(S\cap S')\\
 &\leq\frac{\alpha(S)}{\dim S}
       \bigl(\dim(S+S')+\dim(S\cap S')\bigr)\\
 &=\frac{\alpha(S)}{\dim S}(\dim S+\dim S')\\
 &=\alpha(S)+\alpha(S').
 \end{aligned}
\]
Every inequality is an equality. In particular, $S+S'$ attains the same maximal ratio. As $S\subset S+S'$, the maximality of $\dim S$ forces $S+S'=S$. Since $\dim S'=\dim S$, we obtain $S'=S$.

This holds for every $h_0\in G$, so $S$ is a nonzero $G$-invariant subspace. Irreducibility gives $S=V$. The maximal ratio is consequently
\[
 \frac{\alpha(V)}{\dim V}=\frac{\dim B}{\dim V}.
\]
Thus, for every nonzero subspace $A\subset V$,
\[
 \frac{\alpha(A)}{\dim A}\leq \frac{\alpha(V)}{\dim V}=\frac{\dim B}{\dim V}.
\]
This proves \eqref{eq:generic-intersection} for $A\ne0$; the case $A=0$ is immediate.

We next analyze equality in the bound just proved. The subspaces of $V$ for which equality holds in
\eqref{eq:generic-intersection} are closed under sums,
intersections and $G$-translation. Suppose that there is a nonzero proper subspace in this set. Choose one of minimal positive dimension and call it $S_1$. The span of all its $G$-translates is a nonzero $G$-invariant subspace, and hence is $V$. We construct a direct sum decomposition of $V$ using these
translates. If $S_1\oplus\cdots\oplus S_j$ has already been chosen and is not $V$, choose a translate $S_{j+1}$ of $S_1$ not contained in $S_1+\cdots+S_j$. The intersection
\[
 S_{j+1}\cap(S_1+\cdots+S_j)
\]
also makes ~\eqref{eq:generic-intersection} an equality. It is a proper subspace of $S_{j+1}$, whose dimension equals $\dim S_1$. The minimal choice of $S_1$ therefore forces this intersection to be zero. Continuing, we obtain a decomposition into subspaces
\begin{equation}\label{eq:equality-decomposition}
 V=S_1\oplus\cdots\oplus S_k,\qquad k\geq2,
\end{equation}
where every $S_j$ is a translate of $S_1$ and satisfies equality.

Choose $h$ generally for the finitely many subspaces $S_j$. Then
\[
 \sum_j\dim(S_j\cap\rho(h)B)
 =\sum_j\alpha(S_j)
 =\frac{\dim B}{\dim V}\sum_j\dim S_j
 =\dim B.
\]
The sum of the intersections is direct by
\eqref{eq:equality-decomposition}, and is contained in $\rho(h)B$.
The equality of dimensions gives
\[
 \rho(h)B=\bigoplus_j(S_j\cap\rho(h)B).
\]
Consequently, every projection $p_j:V\to S_j\subset V$ associated
with \eqref{eq:equality-decomposition} preserves $\rho(h)B$ for
general $h$.

For a fixed $p_j$, the condition $p_j(\rho(h)B)\subset\rho(h)B$
is closed in $h$: it is equivalent to the vanishing of the composite
\[
 B\xrightarrow{\rho(h)}V\xrightarrow{p_j}V
  \xrightarrow{\rho(h)^{-1}}V\longrightarrow V/B.
\]
Thus every $p_j$ preserves every $\rho(h)B$ by the connectedness of $G$. Since $0\neq S_j\neq V$, these projections
are nonzero, different from the identity, and non-scalar.

Assume from now on that $G$ is almost simple and $0\neq B\subsetneq V$. Set
\[
 \mathcal A= \{u\in\End(V):u(\rho(h)B)\subset\rho(h)B  \text{ for every }h\in G\}.
\]
It is a finite-dimensional unital associative subalgebra of $\End(V)$. It is proper, since every one of its elements preserves $B$, whereas not every endomorphism of $V$ preserves the nonzero proper subspace $B$. Moreover,
\begin{equation}\label{eq:intersection-algebra-normalization}
 \rho(h_0)\mathcal A\rho(h_0)^{-1}=\mathcal A  \qquad(h_0\in G).
\end{equation}
Indeed, if $u\in\mathcal A$, then
\[
 (\rho(h_0)u\rho(h_0)^{-1})(\rho(h)B) =\rho(h_0)u(\rho(h_0^{-1}h)B) \subset\rho(h)B.
\]
This proves one inclusion in \eqref{eq:intersection-algebra-normalization}; replacing $h_0$ by $h_0^{-1}$ proves the other.

\smallskip\noindent\textbf{Claim.}
\[
  \mathcal A=\C I_V.
\]

Assume the claim. If a nonzero proper subspace attains equality
in \eqref{eq:generic-intersection}, the construction above gives
a decomposition \eqref{eq:equality-decomposition} whose nontrivial
projections preserve every $\rho(h)B$. These projections lie in
$\mathcal A$, contradicting $\mathcal A=\C I_V$. This proves
the strict inequality.

Now we prove the claim. Write $\operatorname{rad}\mathcal A$ for the Jacobson radical of the
associative algebra $\mathcal A$. We use the standard facts that this
radical is a nilpotent two-sided ideal preserved by all algebra
automorphisms, and that a finite-dimensional complex algebra with zero
radical is a product of full matrix algebras; see
\cite[Proposition~3.5.3 and Theorem~3.5.4]{EGH11}.
By \eqref{eq:intersection-algebra-normalization}, the subspace
\[
 (\operatorname{rad}\mathcal A)V
 =\operatorname{span}\{uv:
       u\in\operatorname{rad}\mathcal A,\ v\in V\}
\]
is $G$-invariant. Irreducibility makes it either zero or all of $V$.
If it were $V$, induction would give
$(\operatorname{rad}\mathcal A)^mV=V$ for every $m\geq1$,
contradicting nilpotence. Thus $(\operatorname{rad}\mathcal A)V=0$, so every element of
$\operatorname{rad}\mathcal A\subset\End(V)$ is the zero
endomorphism. Hence
\[
 \operatorname{rad}\mathcal A=0.
\]

By the structure theorem just recalled,
\[
 \mathcal A\simeq
 \prod_{i=1}^m \End(W_i)
\]
for some finite collection of vector spaces $W_i$. Hence the central idempotents of $\mathcal A$ form a finite set. Conjugation by $G$ preserves this set, and since $G$ is connected, every central idempotent is fixed by conjugation.

If $e$ is such an idempotent, then $eV$ is a $G$-invariant subspace. Irreducibility implies that $eV$ is either $0$ or $V$. Since $\mathcal A\subset\End(V)$ acts faithfully, this means $e=0$ or $e=I_V$. Thus $\mathcal A$ has no nontrivial central idempotent, and consequently $m=1$. Write $W_1=V^\prime$. The structure of modules over a full matrix algebra then gives
\begin{equation}\label{eq:intersection-algebra-tensor}
 V\simeq V'\otimes V'',\qquad  \mathcal A=\End(V')\otimes I_{V''}.
\end{equation}
Suppose that $\mathcal A\neq\C I_V$. Then $\dim V'>1$. Since $\mathcal A$ is proper in $\End(V)$, we also have $\dim V''>1$.

Differentiating \eqref{eq:intersection-algebra-normalization} gives
\[
 [d\rho(x),a\otimes I_{V''}]
 \in\End(V')\otimes I_{V''}
 \qquad(x\in\Lie G,\ a\in\End(V')).
\]
Choose a basis $(b_j)_j$ of the trace-free endomorphisms
of $V''$. Since $\End(V)=\End(V')\otimes \End(V'')$, we can write
\[
 d\rho(x)=a_0(x)\otimes I_{V''}
          +\sum_j a_j(x)\otimes b_j,
\]
with $a_0(x), a_j(x)\in \End(V')$. Comparing the coefficients of $b_j$ in the preceding commutator shows that
\[
 [a_j(x),a]=0
 \qquad\text{for every }a\in\End(V').
\]
Thus each $a_j(x)$ is scalar. Since $G$ is almost simple, $\Lie G$ is perfect, and hence
$\operatorname{tr}d\rho(x)=0$ for every $x\in\Lie G$. As the $b_j$ are trace free, $a_0(x)$ is trace free as well.
Consequently there are unique trace free endomorphisms $\rho'(x)\in \End(V')$ and $\rho''(x)\in \End(V'')$ such that
\begin{equation}\label{eq:intersection-Lie-tensor-action}
 d\rho(x)=\rho'(x)\otimes I_{V''} +I_{V'}\otimes\rho''(x)\in \bigl(\mathfrak{sl}(V')\otimes I_{V''}\bigr) \oplus \bigl(I_{V'}\otimes\mathfrak{sl}(V'')\bigr).
\end{equation}
The two summands commute and have zero intersection. Since \(d\rho\) is a Lie algebra homomorphism, for \(x,y\in\Lie G\) we have
\[
 d\rho([x,y])=[d\rho(x),d\rho(y)].
\]
Comparing the two summands gives
\[
 \rho'([x,y])=[\rho'(x),\rho'(y)],
 \qquad
 \rho''([x,y])=[\rho''(x),\rho''(y)].
\]
Thus \(\rho'\) and \(\rho''\) are Lie algebra representations.

Since $G$ is connected, $V$ is also irreducible as a $\Lie G$-module. Neither $\rho'$ nor $\rho''$ can be zero: if $\rho'=0$, then $W\otimes V''$ is invariant for every subspace $W\subset V'$, contradicting irreducibility and $\dim V'>1$; the other case is identical. Simplicity of $\Lie G$ therefore makes both homomorphisms injective.

Choose Killing dual bases $x_a,x^a$ of $\Lie G$. The operator
\begin{equation}\label{eq:intersection-mixed-Casimir-nonzero}
 \sum_a\rho'(x_a)\otimes\rho''(x^a)
\end{equation}
is nonzero: the inverse Killing tensor is nonzero, and $\rho'\otimes\rho''$ is injective. It has trace zero because $\rho'$ takes values in trace free endomorphisms. Invariance of the Killing tensor and \eqref{eq:intersection-Lie-tensor-action} show that it commutes with $d\rho(\Lie G)$. Since $V$ is irreducible, Schur's lemma makes it scalar. Also its zero trace then makes it zero, a contradiction. Hence
\begin{equation}\label{eq:intersection-preserver-algebra}
 \mathcal A=\C I_V.
\end{equation}
This proves the claim.
\end{proof}

\begin{lemma}\label{lem:simultaneous-strictness}
Let $G_{\mathrm{ad}}=\prod_iG_i$ and $V=\bigotimes_iV_i$ as in \eqref{eq:tensor-decomp-intro}, and let $I$ be the set of indices $i$ for which $V_i$ is nontrivial. For each $i\in I$ choose $0\ne B_i^0\subsetneq V_i$ and put
\[
 B_i=B_i^0\otimes\bigotimes_{k\ne i}V_k.
\]
For every $0\ne A\subsetneq V$ and general $h\in G_{\mathrm{ad}}$,
\[
 \dim(A\cap\rho(h)B_i)\leq\frac{\dim A\,\dim B_i}{r}
 \quad(i\in I),
\]
and at least one inequality is strict.
\end{lemma}

\begin{proof}
Lemma~\ref{lem:one-factor-strict} gives the non-strict
inequalities. Suppose equality holds for every $i\in I$. The family of subspaces
attaining equality for every $i$ is again closed under sums, intersections and translates: apply \eqref{eq:intersection-supermodular} separately for each $B_i$.

The same construction as in~\eqref{eq:equality-decomposition} gives a nontrivial direct sum decomposition of $V$ whose projections preserve every translate of every $B_i$.

For fixed $i$, the proof of Lemma~\ref{lem:one-factor-strict} applied on $V_i$ gives
\[
 \{u\in\End(V_i):u(\rho_i(h_i)B_i^0)\subset\rho_i(h_i)B_i^0
                 \text{ for all }h_i\}=\C I_{V_i}.
\]
Expanding an endomorphism of $V$ in a basis of the endomorphisms of the other tensor factors therefore shows that the algebra preserving all translates of $B_i$ is
\[
 \C I_{V_i}\otimes\End\!\left(\bigotimes_{k\ne i}V_k\right).
\]
The intersection of these algebras for all $i\in I$ is $\C I_V$ since factors $\End(V_j)$ for $j\notin I$ are one dimensional. It cannot contain the nontrivial projections just constructed. This contradiction proves the lemma.
\end{proof}

\section{\texorpdfstring{The genus two case for $\SU_C(n, L)$}{The genus two case for SU(C,n,L)}}\label{sec:examples}

Let $g=2$, $M=\SU_C(n,L)$, $d=\deg L$ and $\delta=\gcd(n,d)$.
Take $G=\GL_n$, $V=\mathfrak{sl}_n$ and
$\rho(h)=\Ad(h)|_{\mathfrak{sl}_n}$. Then $r=n^2-1$ and
$E_\rho(P)=\End_0E$, where $P$ is the frame bundle of $E$.
Here $\MM$ is the stack of rank-$n$ bundles $E$ with a specified
isomorphism $\det E\simeq L$; write $\MM^s$ for its stable
substack and $M^s\subset M$ for the stable locus.
The morphism $\MM^s\to M^s$ is a $\mu_n$-gerbe, and these scalar
automorphisms act trivially on $\End_0E$. The proof of
Proposition~\ref{prop:associated-descent} therefore gives the
associated bundle $\EE_\rho$ on $M^s\times C$. The deformation
calculation in the proof of Proposition~\ref{prop:inputs}, with
the determinant fixed, gives
\begin{equation}\label{eq:genus2-tangent-space}
 (T_M)_{[E]}\simeq H^1(C,\End_0E)\qquad([E]\in M^s).
\end{equation}
For $p:M^s\times C\to M^s$, these identifications arise from the
deformation-theoretic isomorphism
$R^1p_*\EE_\rho\simeq T_M|_{M^s}$.
We do not require the associated $\PGL_n$-bundle to be regularly
stable. In this section, highest-root Hecke curves are
understood as in Definition~\ref{def:highest-root-hecke}, with
$\mathfrak U$ and $U$ replaced by $\MM^s$ and $M^s$.
By \cite[Theorems~A, B and~F]{DN89}, $M$ is normal and
locally factorial, and its ample generator $H$ satisfies
\begin{equation}\label{eq:genus2-index}
 \Pic(M)=\mathbb Z H,\qquad \det T_M=\OO_M(2\delta H).
\end{equation}
The sheaf $T_M=\mathcal Hom_{\OO_M}(\Omega_M^1,\OO_M)$ is
reflexive by \cite[Tag~0AY4]{Stacks}.
Theorem~\ref{thm:interp} and Lemma~\ref{lem:adjoint-reduction},
with $g-1=1$, give, for general $(E,x)$,
\begin{equation}\label{eq:genus2-interpolation}
 H^0(C,(\End_0E)(x))=H^1(C,(\End_0E)(x))=0.
\end{equation}
The family \eqref{eq:hecke-uniformization} is defined for every genus,
and Remark~\ref{rem:splitting-beyond} applies to the construction of
$\Lambda^+$ in \eqref{eq:superlattice} and of $E^+$. In genus two,
we must check that the resulting curves can lie entirely in $M^s$
and that $H^0(C,E^+)=0$. The latter no longer follows from
$E^+\subset(\End_0E)(2x)$, since the adjoint representation has
$m_\rho=2>g-1$. The same inequality prevents us from applying
\eqref{eq:weighted-i}--\eqref{eq:nonstrict-i}.

We will instead prove $0<\deg f^*\FF<n$ for any saturated subsheaf
with $\mu_H(\FF)\geq\mu_H(T_M)$. After excluding equality of
slopes, we use the action of $\Pic^0(C)[n]$ to show that the degree
of the first Harder--Narasimhan subsheaf is divisible by $n$.
Until the final paragraph of this section, assume that $n\geq3$,
or that $n=2$ and $d$ is odd.

\begin{lemma}\label{lem:genus2-hecke}
The conclusion of Proposition~\ref{prop:complete-hecke} holds with
$U$ replaced by $M^s$.
\end{lemma}

\begin{proof}
The morphism \eqref{eq:framed-hecke-evaluation} takes values in
$\MM$, not only in $\MM^s$. To apply the proof of
Proposition~\ref{prop:complete-hecke} with $\MM^s$ in place of
$\mathfrak U$, we must verify
\[
 \operatorname{codim}_{\MM}(\MM\setminus\MM^s)\geq2.
\]
The dimension calculation for the fixed-determinant stack in
\cite[proof of Corollary~3.2]{Hof12} gives the following bound:
the locus admitting a rank-$k$ subbundle of degree $a$ has
codimension at least
\[
 k(n-k)+na-kd.
\]
For a nonstable bundle, such a subbundle exists with $na-kd\geq0$.
If $n\geq3$, the bound is at least $n-1\geq2$. If $n=2$ and
$d$ is odd, then $2a-d\geq1$, so the bound is again at least two.
Only finitely many destabilizing ranks and degrees occur on each
bounded chart, proving the required codimension bound.
We can now apply the proof of Proposition~\ref{prop:complete-hecke}
with $\MM^s$ in place of $\mathfrak U$ and $M^s$ in place of $U$.
The proof of Proposition~\ref{prop:associated-descent} applies to
$\MM^s\to M^s$ as stated at the start of this section.
By \eqref{eq:genus2-interpolation} and the cohomology sequence
of \eqref{eq:principal-parts-exact-sequence}, the connecting map
\[
 H^0\bigl(C,(\End_0E)(x)/\End_0E\bigr)
 \longrightarrow H^1(C,\End_0E)
\]
is injective. Together with \eqref{eq:genus2-tangent-space}, this
verifies the hypotheses used in the last nonconstancy calculation
in the proof of Proposition~\ref{prop:complete-hecke}.
\end{proof}

For a curve with $f(0)=[E]$, fix a uniformizer $z$ at $x$ and put
$D_x=\operatorname{Spec}\C[[z]]$. Choose a trivialization of
$L|_{D_x}$ and a trivialization
\begin{equation}\label{eq:genus2-source-frame}
 E|_{D_x}\simeq\bigoplus_{j=1}^n\OO_{D_x}\mathbf e_j
\end{equation}
whose determinant induces the chosen trivialization of $L|_{D_x}$,
as in \eqref{framing}. Choose $e_{\theta_1}=e_{1n}$, sending
$\mathbf e_n$ to $\mathbf e_1$ and killing the other basis vectors,
and put $\nu=\nu_1=\ad(e_{1n})$.

\Needspace{8\baselineskip}
\begin{lemma}\label{lem:genus2-adjoint-splitting}
For a general curve in Lemma~\ref{lem:genus2-hecke}, the bundle $E^+$
constructed from \eqref{eq:superlattice} in the proof of
Lemma~\ref{lem:common-containing-bundle} satisfies $H^0(C,E^+)=0$.
\end{lemma}

\begin{proof}
The lattice \eqref{eq:superlattice} is
\begin{equation}\label{eq:genus2-adjoint-superlattice}
 \Lambda^+=\mathfrak{sl}_n[[z]]+
 z^{-1}\im\nu[[z]]+z^{-2}\C[[z]]e_{1n}.
\end{equation}
Let $E'$ be the elementary modification with local lattice
\[
 \C[[z]]z^{-1}\mathbf e_1\oplus
 \bigoplus_{j=2}^n\C[[z]]\mathbf e_j,
 \qquad 0\longrightarrow E\longrightarrow E'
 \longrightarrow\C_x\longrightarrow0.
\]
The matrices in $\im\nu$ have no entries in column $1$ below the
diagonal. Comparing pole orders therefore gives
\[
 \Lambda^+\subset
 z^{-1}\Ad\!\bigl(\operatorname{diag}(z^{-1},1,\ldots,1)\bigr)
 \mathfrak{sl}_n[[z]],
\]
and hence
\begin{equation}\label{eq:genus2-containing-inclusion}
 \End_0E\subset E^+\subset(\End_0E')(x).
\end{equation}
We have $\det E'\simeq L(x)$. The proof of
Proposition~\ref{prop:inputs} establishes the surjectivity of
\eqref{eq:adjoint-map} for $g\geq2$, as noted immediately
after that proof. Thus, for each $x$, every $\PGL_n$-bundle of
topological type $d+1\bmod n$ has a vector-bundle lift with
determinant $L(x)$. Hence the morphism
\[
 (E',x)\longmapsto(\mathbb P(E'),x)
 \in\MM_{\PGL_n}^{d+1\bmod n}\times C
\]
from the stack of pairs with $\det E'\simeq L(x)$ is surjective.
Applying Theorem~\ref{thm:interp} on the target and the
identification \eqref{eq:associated-adjoint} from
Lemma~\ref{lem:adjoint-reduction}, we obtain
\begin{equation}\label{eq:genus2-modification-vanishing}
 H^0\bigl(C,(\End_0E')(x)\bigr)=0
\end{equation}
on a nonempty open subset of these pairs.
The correspondence of inclusions $E\subset E'$ with
$E'/E\simeq\C_x$ has fibers $\Pp^{n-1}$ over each stack of
pairs: over $(E,x)$ it parametrizes lines in $E_x$, and over
$(E',x)$ it parametrizes one-dimensional quotients of $E'_x$.
These descriptions make both projections projective bundles
and hence surjective. The correspondence is irreducible because
the stack of $(E,\det E\simeq L,x)$ is irreducible by the proof
of Proposition~\ref{prop:inputs}.

Use $\mathscr S_a$ from the proof of
Proposition~\ref{prop:complete-hecke} for tuples
$(E,\det E\simeq L,x,\psi)$, where $\psi$ is the reduction of
\eqref{eq:genus2-source-frame} modulo $z^a$. The map sending
such a tuple to the modification determined by the line
$\C\mathbf e_1\subset E_x$ is surjective: any line in $E_x$
can be the first line of a frame inducing the prescribed
trivialization of $\det E$. Consequently,
\eqref{eq:genus2-modification-vanishing} holds on a nonempty
open subset of $\mathscr S_a$. Apply the proof of
Proposition~\ref{prop:complete-hecke}, as in
Lemma~\ref{lem:genus2-hecke}, after restricting to this open
subset. Its smooth evaluation maps and codimension estimate
still produce complete curves in $M^s$ avoiding any prescribed
closed subset of codimension at least two. For these curves,
\eqref{eq:genus2-containing-inclusion} and
\eqref{eq:genus2-modification-vanishing} imply $H^0(C,E^+)=0$.
\end{proof}

Cohomology and base change applied to
$R^1p_*\EE_\rho\simeq T_M|_{M^s}$ give
$R^1\pi_*\EE_f\simeq f^*T_M$.
By Lemma~\ref{lem:genus2-adjoint-splitting}, $H^0(C,E^+)=0$.
Pushing forward \eqref{eq:Heckeexact} therefore gives
\eqref{eq:hecke-cohomology-sequence} with
$f^*\mathcal V_\rho^\circ$ replaced by $f^*T_M$.
Remark~\ref{rem:splitting-beyond} allows us to apply
Lemma~\ref{lem:sl2-local-quotient} and
Proposition~\ref{prop:splitting}. Substituting
$\rk\nu=2n-2$, $\im\nu^2=\C e_{1n}$ and $\nu^3=0$ into
\eqref{eq:splitting} and \eqref{eq:Tj} gives
\begin{equation}\label{eq:genus2-adjoint-splitting}
 f^*T_M\simeq\OO_{\Pp^1}(2)\oplus
 \OO_{\Pp^1}(1)^{\oplus(2n-2)}\oplus
 \OO_{\Pp^1}^{\oplus n(n-2)},
\end{equation}
with
\begin{equation}\label{eq:genus2-positive-fibers}
 T_1=\delta_0\bigl(z^{-1}\im\nu\oplus z^{-2}\C e_{1n}\bigr),
 \qquad T_2=\delta_0\bigl(z^{-1}\C e_{1n}\bigr).
\end{equation}
Here $T_j=T_{1,j}=B_{1,\geq j}|_0$ is defined by
\eqref{eq:positive-filtration-definition} with $f^*T_M$ in place
of $f^*\mathcal V_\rho^\circ$, and $\delta_0$ is
the map \eqref{eq:source-connecting}, which is injective by
Lemma~\ref{lem:genus2-adjoint-splitting}. Taking degrees in
\eqref{eq:genus2-adjoint-splitting} and using
\eqref{eq:genus2-index} gives
\begin{equation}\label{eq:genus2-hecke-degree}
 \deg f^*T_M=2n,\qquad \deg f^*\OO_M(H)=n/\delta.
\end{equation}

\begin{proof}[Proof of Theorem~\ref{thm:tangent} for $g=2$]
Suppose $0\ne\FF\subsetneq T_M$ is saturated and
$\mu_H(\FF)\geq\mu_H(T_M)$.
Choose a general pair $(E,x)$ satisfying
\eqref{eq:genus2-interpolation}. By Lemmas~\ref{lem:genus2-hecke}
and~\ref{lem:genus2-adjoint-splitting}, choose the trivialization
\eqref{eq:genus2-source-frame} so that its Hecke curve
$f:\Pp^1\to M^s$ avoids the non-locally-free loci of $\FF$ and
$T_M/\FF$ and the corresponding $E^+$ has $H^0(C,E^+)=0$.
By \eqref{eq:genus2-index} and \eqref{eq:genus2-hecke-degree},
$\Pic(M)=\mathbb ZH$ and $\deg f^*\OO_M(H)=n/\delta$.
Therefore the degrees of $f^*\FF$ and
$f^*(T_M/\FF)$ do not depend on the choice of curve. They satisfy
\[
 \deg f^*\FF+\deg f^*(T_M/\FF)=\deg f^*T_M=2n.
\]
Set
\[
 \begin{gathered}
 Z=\Ann(\FF|_{[E]})\subset H^0(C,\End_0E\otimes K_C),\\
 Z^\vee=(T_M/\FF)|_{[E]},\qquad \dim Z=\rk(T_M/\FF).
 \end{gathered}
\]
By \eqref{eq:genus2-interpolation} and Serre duality, evaluation is
an isomorphism
\begin{equation}\label{eq:genus2-evaluation}
 H^0(C,\End_0E\otimes K_C)
 \xrightarrow{\sim}(\End_0E\otimes K_C)_x.
\end{equation}
Use the trivialization \eqref{eq:genus2-source-frame}, $dz|_x$
and the trace pairing to identify the image of $Z$ with
$Z_x\subset\mathfrak{sl}_n^\vee$. Then
\[
 \dim Z_x=\rk(T_M/\FF),\qquad Z(-x)=0.
\]
Thus the restriction of \eqref{eq:genus2-evaluation} is an
isomorphism $Z\xrightarrow{\sim}Z_x$; we will use it directly
in the residue pairing \eqref{eq:residue-pairing}.
Keep $E,x,z$ and the trivialization
\eqref{eq:genus2-source-frame} fixed. For $h\in\PGL_n$, conjugate
the homomorphism \eqref{eq:root-loop-group} by a lift of $h$
to $\SL_n$. On the affine chart \eqref{eq:hecke-affine-chart},
this replaces
\[
 \exp(te_{1n}/z)
 \quad\text{by}\quad
 \exp\bigl(t\Ad(h)e_{1n}/z\bigr).
\]
Different lifts give the same conjugation. In the fixed
trivialization, \eqref{eq:genus2-adjoint-superlattice} becomes
$\Ad(h)\Lambda^+$, whereas $Z_x$ is unchanged. Gluing this
lattice to $\End_0E|_{C\setminus\{x\}}$, as in the proof of
Lemma~\ref{lem:common-containing-bundle}, gives a bundle that
we continue to denote by $E^+$. These bundles form an algebraic
family over $\PGL_n$. By semicontinuity
\cite[Tag~0BDN]{Stacks}, $H^0(C,E^+)=0$ for $h$ in an open
neighbourhood of $1$.

The conjugated homomorphisms also give an algebraic family
\[
 \PGL_n\times\Pp^1\longrightarrow\MM
\]
by \eqref{eq:hecke-uniformization}. The open substack of
$\MM^s$ on which both $\FF$ and $T_M/\FF$ are locally free
has closed complement in $\MM$. The inverse image of this
complement is closed in $\PGL_n\times\Pp^1$, and its
projection to $\PGL_n$ is closed because $\Pp^1$ is proper.
It does not contain $1$, by the choice of $f$. Intersecting
its complement with the preceding neighbourhood of $1$, we
obtain a nonempty open subset on which every resulting curve
lies in $M^s$, avoids both non-locally-free loci, and satisfies
$H^0(C,E^+)=0$.

By \eqref{eq:hecke-affine-chart}, its tangent vector at $0$ is
the connecting image of $z^{-1}\Ad(h)e_{1n}$. This image is
nonzero by \eqref{eq:genus2-interpolation},
\eqref{eq:principal-parts-exact-sequence}, and
\eqref{eq:genus2-tangent-space}. For a curve in this open
subset, continue to write $f$ for the curve and use
$T_j=B_{1,\geq j}|_0$ from
\eqref{eq:positive-filtration-definition}. Put
\[
 t_j=\rk(T_j\longrightarrow Z^\vee)\qquad(j=1,2).
\]
The degree estimate \eqref{eq:factor-degree-bound}, applied to
\eqref{eq:genus2-adjoint-splitting}, gives
\[
 \deg f^*(T_M/\FF)\geq t_1+t_2.
\]
Use the fixed principal-parts connecting map
\[
 \partial^{\End_0E}_{x,2}:
 H^0\bigl(C,(\End_0E)(2x)/\End_0E\bigr)
 \longrightarrow H^1(C,\End_0E).
\]
On $\Ad(h)\Lambda^+/\Lambda_0$, this map is
\eqref{eq:source-connecting} for the chosen curve and is
injective because $H^0(C,E^+)=0$; see
Remark~\ref{rem:splitting-beyond}. We use injectivity only
on this subspace.
For the curve defined by $h$, \eqref{eq:genus2-positive-fibers} gives
\[
 \partial^{\End_0E}_{x,2}\bigl(z^{-1}\Ad(h)\im\nu\bigr)\subset T_1,
 \qquad
 T_2=\partial^{\End_0E}_{x,2}\bigl(z^{-1}\Ad(h)\C e_{1n}\bigr).
\]
Restricting the pairing $\langle T_1,Z\rangle$ to the displayed
subspace of $T_1$ can only decrease its rank. By the residue formula
\eqref{eq:residue-pairing}, this restricted pairing is
\[
 Z_x\times\im\nu\longrightarrow\C,
 \qquad (u,v)\longmapsto u(\Ad(h)v).
\]
The strict part of Lemma~\ref{lem:one-factor-strict}, applied to
$Z_x^\perp$ and $\im\nu$ in the irreducible $\PGL_n$-module
$\mathfrak{sl}_n$, gives, for general $h$,
\[
 t_1\geq\rk\bigl(Z_x\longrightarrow(\Ad(h)\im\nu)^\vee\bigr)
       >\frac{(2n-2)\rk(T_M/\FF)}{r}.
\]
For $T_2$, the non-strict inequality of
Lemma~\ref{lem:one-factor-strict}, applied to $Z_x^\perp$ and
$\C e_{1n}$, gives
\[
 \begin{aligned}
 t_2 =\rk\bigl(Z_x\longrightarrow(\Ad(h)\C e_{1n})^\vee\bigr), \qquad
 1 \geq t_2\geq\frac{\rk(T_M/\FF)}{r}>0.
 \end{aligned}
\]
Since $t_2$ is an integer, $t_2=1$. Choose $h$ in these dense open
rank loci and in the nonempty open subset constructed above.
Consequently
\begin{equation}\label{eq:genus2-quotient-refinement}
 \deg f^*(T_M/\FF)\geq t_1+t_2>
 \frac{(2n-2)\rk(T_M/\FF)}{r}+1.
\end{equation}
This is the quotient-degree inequality used in place of
\eqref{eq:factor-quotient-estimate}.

Since $\mu_H(\FF)\geq\mu_H(T_M)$,
\eqref{eq:genus2-index} and \eqref{eq:genus2-hecke-degree}
give $\deg f^*\FF>0$ and
\[
 \frac{\deg f^*(T_M/\FF)}{\rk(T_M/\FF)}\leq\frac{2n}{r}.
\]
Therefore
\[
 \begin{aligned}
 \deg f^*(T_M/\FF)
 &>\frac{(2n-2)\rk(T_M/\FF)}{r}+1\\
 &\geq\frac{n-1}{n}\deg f^*(T_M/\FF)+1,
 \end{aligned}
\]
so $\deg f^*(T_M/\FF)>n$, or equivalently
\begin{equation}\label{eq:genus2-small-degree}
 0<\deg f^*\FF<n.
\end{equation}
If the slopes were equal, then
$r\deg f^*\FF=2n\rk\FF$. Since $\gcd(n,r)=1$, this would imply
$n\mid\deg f^*\FF$, a contradiction.
Thus failure of stability would imply instability.

It remains to exclude instability. Let $\FF$ now be the first
nonzero term of the Harder--Narasimhan filtration of $T_M$
with respect to $\mu_H$; see \cite[\S1.6]{HL10}.
The group $\Gamma=\Pic^0(C)[n]$ acts by
$[E]\mapsto[E\otimes\gamma]$ for $\gamma\in\Gamma$, and
preserves $H$. Uniqueness of the filtration makes the canonical
$\Gamma$-linearization of $T_M$ restrict to $\FF$, and hence to
$\det\FF=\OO_M(aH)$. By \cite[(6.2), the proof of
Proposition~9.1, and the paragraph following Proposition~10.1]{BLS98},
the commutator pairing associated with $\OO_M(aH)$ is
$e_n^{-an/\delta}$, where $e_n$ is the Weil pairing on $\Gamma$.
A $\Gamma$-linearization makes $e_n^{-an/\delta}$ trivial.
Since the Weil pairing $e_n$ attains a primitive $n$th-root of unity,
\[
 n\mid\frac{an}{\delta},\qquad\text{equivalently}\qquad\delta\mid a.
\]
Equation~\eqref{eq:genus2-hecke-degree} therefore gives
\begin{equation}\label{eq:genus2-HN-degree}
 \deg f^*\FF=an/\delta\in n\mathbb Z,
\end{equation}
contradicting \eqref{eq:genus2-small-degree}.

Finally, if $n=2$ and $d$ is even, tensoring by a line bundle
reduces to trivial determinant, and $M\simeq\Pp^3$ by \cite{NR69}.
The tangent bundle of $\Pp^3$ is slope stable; see
\cite[\S1.4]{HL10}. This proves the assertion for all $n$ and $L$.
\end{proof}

\bibliographystyle{alpha}
\bibliography{representation_stability_restructured}
\end{document}